\documentclass[11pt]{article}
\newif\ifnotlms \notlmstrue
\usepackage{amssymb,amsmath,amsthm,verbatim}
\usepackage{fullpage}
\usepackage{xcolor}
\usepackage{url}
\usepackage{mathrsfs}
\usepackage[all]{xy}
  \SelectTips{cm}{10}
  \everyxy={<2.5em,0em>:}
\usepackage{tikz}
\usepackage{picinpar} 

\usepackage{fancyhdr}
\usepackage{ifpdf}
  \ifpdf
    \usepackage{hyperref} 
  \else
    \newcommand{\texorpdfstring}[2]{#1}
    \newcommand{\href}[2]{#2}
  \fi

\theoremstyle{plain}
  \newtheorem{lemma}[equation]{Lemma}
  \newtheorem{proposition}[equation]{Proposition}
  \newtheorem{theorem}[equation]{Theorem}
  \newtheorem{corollary}[equation]{Corollary}
    \newtheorem{question}[equation]{Question}

\theoremstyle{definition}
  \newtheorem{definition}[equation]{Definition}
  \newtheorem{notation}[equation]{Notation}

\theoremstyle{remark}
  \newtheorem{remark}[equation]{Remark}

\renewcommand{\thesection}{\arabic{section}}
\renewcommand{\theequation}{\thesection.\arabic{equation}}

 \DeclareFontFamily{U}{manual}{}
 \DeclareFontShape{U}{manual}{m}{n}{ <->  manfnt }{}
 \newcommand{\manfntsymbol}[1]{%
    {\fontencoding{U}\fontfamily{manual}\selectfont\symbol{#1}}}

\makeatletter
   \@addtoreset{section}{part}
   \@addtoreset{equation}{section}
   \@addtoreset{footnote}{section}

    {\hspace*{\fill}$\lrcorner$\endgraf\endgroup\end{trivlist}}
 \newenvironment{example}[1][]{
   \refstepcounter{equation}
   \begin{proof}[Example~\theequation%
   \@ifnotempty{#1}{ (#1)}.]
   }
  {\end{proof}}
\makeatother

  \DeclareFontFamily{OT1}{pzc}{}
  \DeclareFontShape{OT1}{pzc}{m}{it}{<-> s * [1.100] pzcmi7t}{}
  \DeclareMathAlphabet{\mathpzc}{OT1}{pzc}{m}{it}

\newif\ifhascomments \hascommentstrue
\ifhascomments
  \newcommand{\anton}[1]{{\color{red}[[\ensuremath{\bigstar\bigstar\bigstar} #1]]}}
  \newcommand{\matt}[1]{{\color{red}[[\ensuremath{\spadesuit\spadesuit\spadesuit} #1]]}}
\else
  \newcommand{\anton}[1]{}
  \newcommand{\matt}[1]{}
\fi

\renewcommand{\AA}{\mathbb{A}}

\DeclareMathOperator{\Aut}{\ensuremath{\mathcal{A}\kern-.125em\mathpzc{ut}}}

\newcommand{\can}{\mathrm{can}}
\DeclareMathOperator{\cok}{cok}

\newcommand{\CC}{\mathbb C}
\DeclareMathOperator{\CaCl}{CaCl}
\DeclareMathOperator{\Cl}{Cl}

\newcommand{\D}{\mathcal D}

\newcommand{\E}{\mathcal E}
\DeclareMathOperator{\Endo}{\ensuremath{\mathcal{E}\kern-.125em\mathpzc{nd}}}

\newcommand{\GG}{\mathbb G}

\let\hom\relax
\DeclareMathOperator{\hom}{Hom}
\newcommand{\hhat}[1]{\widehat{#1}}
\DeclareMathOperator{\Hom}{\ensuremath{\mathcal{H}\kern-.125em\mathpzc{om}}}
\newcommand{\id}{\mathrm{id}}

\newcommand{\K}{\mathcal{K}}

\renewcommand{\L}{\mathcal L}
\newcommand{\m}{\mathfrak m}
\newcommand{\M}{\mathcal M}
\newcommand{\N}{\mathcal N}

\renewcommand{\O}{\mathcal O}
\DeclareMathOperator{\pic}{Pic}

\newcommand{\PP}{\mathbb{P}}

\newcommand{\QQ}{\mathbb Q}
\newcommand{\RR}{\mathbb R}
\DeclareMathOperator{\rk}{rk}

\renewcommand{\setminus}{\smallsetminus}

\DeclareMathOperator{\spec}{Spec}
\DeclareMathOperator{\stab}{Stab}
\DeclareMathOperator{\sym}{Sym}

\newcommand{\U}{\mathcal U}
\newcommand{\V}{\mathcal V}
\newcommand{\VV}{\mathbb V}

\newcommand{\X}{\mathcal{X}}
\newcommand{\Y}{\mathcal{Y}}
\newcommand{\Z}{\mathcal{Z}}
\newcommand{\ZZ}{\mathbb{Z}}

\newcommand{\charp}[1]{#1}

 \def\ari[#1]{\ar@{^(->}[#1]}
 \def\are[#1]{\ar[#1]^{\txt{\'et}}}
 \def\areh[#1]{\ar[#1]|{\txt{$H$-eq}}^{\txt{\'et}}}
 \def\ars[#1]{\ar@{->>}[#1]}
 \newcommand{\dplus}{\ar@{}[d]|{\mbox{$\oplus$}}}
 \newcommand{\dtimes}{\ar@{}[d]|{\mbox{$\times$}}} 

\begin{document}
\title{Torus Quotients as Global Quotients by Finite Groups}
\author{Anton Geraschenko\thanks{\texttt{geraschenko@gmail.com}}\ \ and
Matthew Satriano\thanks{Supported by NSF grant DMS-0943832 and an NSF postdoctoral fellowship (DMS-1103788).}}
\date{}
\maketitle
\begin{abstract}
  This article is motivated by the following local-to-global question: is every variety with tame quotient singularities \emph{globally} the quotient of a smooth variety by a finite group? We show that the answer is ``yes'' for quasi-projective varieties which are expressible as a quotient of a smooth variety by a split torus (e.g.~quasi-projective simplicial toric varieties). Although simplicial toric varieties are rarely \emph{toric} quotients of smooth varieties by finite groups, we give an explicit procedure for constructing the quotient structure using toric techniques.

  The main result follows from a characterization of varieties which are expressible as the quotient of a smooth variety by a split torus. As an additional application of this characterization, we show that a variety with \emph{abelian} quotient singularities may fail to be a quotient of a smooth variety by a finite \emph{abelian} group. Concretely, we show that $\PP^2/A_5$ is not expressible as a quotient of a smooth variety by a finite abelian group.
\end{abstract}

\tableofcontents

\section{Introduction}
\label{sec:intro}
In this paper, we investigate a local-to-global question concerning quotient singularities. Recall that a variety $X$ over a field $k$ is said to have \emph{\charp{tame} (abelian) quotient singularities} if it is \'etale locally the quotient of a smooth variety by a finite (abelian) group \charp{whose order is relatively prime to the characteristic of $k$}.\footnote{Alternatively, $X$ has \charp{tame} quotient singularities if all complete local rings $\hhat\O_{X,x}$ are isomorphic to the ring of invariants in $k(x)[[t_1,\dots, t_n]]$ under the action of a finite group \charp{of order relatively prime to the characteristic of $k$}.} Every variety of the form $U/G$, where $U$ is smooth and $G$ is a finite group \charp{of order relatively prime to the characteristic of $k$}, has at worst \charp{tame} quotient singularities. The motivation for this paper is whether the converse is true, a question suggested to us by William Fulton:
\begin{question}\label{q:main}
 Let $k$ be an algebraically closed field. If $X$ is a variety over $k$ with tame quotient singularities, does there exist a smooth variety $U$ over $k$ with an action of a finite group $G$ such that $X=U/G$?
\end{question}
We show that the answer is ``yes'' when $X$ is quasi-projective and globally a quotient of a smooth variety by a torus.  In particular, every quasi-projective toric variety (with tame quotient singularities) is a global quotient by a finite abelian group (see Corollary \ref{cor:main}). These results are immediate consequences of Theorem \ref{thm:main-intro}, which characterizes quotients of smooth varieties by finite abelian groups.

\begin{theorem}
\label{thm:main-intro}
  Let $X$ be a quasi-projective variety with tame abelian quotient singularities over an algebraically closed field $k$. The following are equivalent.
  \begin{enumerate}
    \item \label{introitem:finite-quot} $X$ is a quotient of a smooth quasi-projective variety by a finite abelian group.
    \item \label{introitem:torus-quot} $X$ is the geometric quotient (in the sense of \cite{git}) of a smooth quasi-projective variety by a torus acting with finite stabilizers.
    \item \label{introitem:divisors} $X$ has Weil divisors $D_1,\dots, D_r$ whose images generate $\Cl(\hhat\O_{X,x})$ for all closed points $x$ of $X$.
    \item \label{introitem:canonical} the canonical stack over $X$ is a stack quotient of a quasi-projective variety by a torus (see Remark \ref{rmk:for-non-stacky}).
  \end{enumerate}
\end{theorem}

\begin{remark}
  Theorem \ref{thm:main-intro} is a consequence of the more technical Theorem \ref{thm:criterion}, which applies to algebraic spaces over infinite fields (in contrast to quasi-projective varieties over algebraically closed fields).
\end{remark}
\begin{remark}[For non-stack-theorists]\label{rmk:for-non-stacky}
  For those not familiar with canonical stacks, (\ref{introitem:canonical}) says that $X$ can be expressed as a quotient of a smooth variety $U$ by a torus $T$ which acts freely on the preimage of the smooth locus of $X$. In fact, this implies that the stabilizers of the $T$-action are as small as one could hope for over the singular locus of $X$. See \S \ref{subsec:lb}.
\end{remark}
\begin{remark}[For stack-theorists]\label{rmk:for-stacky}
  Those familiar with stacks should not confuse Question \ref{q:main} with the question ``is every smooth tame Deligne-Mumford stack the \emph{stack quotient} of a smooth scheme by a finite group?'' The answer to this question is ``no'', e.g.~the weighted projective stack $\mathbf P(1,2)$ is not such a quotient. The appropriate stacky generalization of Question \ref{q:main} is, ``is every smooth tame Deligne-Mumford stack $\X$ a \emph{relative coarse space} of a stack quotient of a smooth scheme by a finite group?'' Our proof of $(\ref{introitem:torus-quot})\Rightarrow (\ref{introitem:finite-quot})$ shows that the answer to this question is ``yes'' when $\X$ is a stack quotient by a split torus and has quasi-projective coarse space.
\end{remark}

\medskip
We emphasize that even when $X$ is a toric variety, the answer to Question \ref{q:main} is not obvious. We show in Proposition \ref{prop:notquot} that if $X$ is the blow-up of weighted projective space $\PP(1,1,2)$ at a smooth torus-fixed point, it is not possible to present $X$ as a \emph{toric} quotient by a finite group. Nevertheless, we show in \S\ref{sec:explicit} that when $X$ is a toric variety, the proof of our main theorem gives a procedure for constructing $U$ as a (non-toric) slice of a higher-dimensional toric variety (see Theorem \ref{thm:tv}). We demonstrate this procedure for the example of $\PP(1,1,2)$ blown-up at a smooth torus-fixed point in \S\ref{sec:ex}, obtaining it as an explicit quotient of a smooth variety by $\ZZ/2$. \S\S\ref{sec:explicit}--\ref{sec:ex} are not needed for the proof of Theorem \ref{thm:main-intro}, but they show that the proof is constructive and that the construction can be described completely without stacks, even though the proof itself relies on stack-theoretic techniques.

\bigskip
There are several variants of Question \ref{q:main} that one can pose.  For example,
\begin{question} \label{q:abelian}
  If $X$ is a variety over an algebraically closed field with tame \emph{abelian} quotient singularities, then is it of the form $U/G$ with $U$ a smooth variety and $G$ a finite \emph{abelian} group?
\end{question}
We show that the answer is ``no'' with Theorem \ref{thm:P2A5}. The key input is the equivalence of (\ref{introitem:finite-quot}) and (\ref{introitem:canonical}) in Theorem \ref{thm:main-intro}, which shows that if the canonical stack over $X$ is not a \emph{stack quotient} by a torus, then $X$ cannot be expressed as a \emph{scheme quotient} of a smooth scheme by a finite abelian group.
\begin{theorem}\label{thm:P2A5}
  Let $k$ be an algebraically closed field with $\mathrm{char}(k)\nmid 60$, and let $V$ be an irreducible 3-dimensional representation of the alternating group $A_5$.
  \begin{enumerate}
  \item $X=\PP(V)/A_5$ has abelian quotient singularities, but is not a quotient of a smooth variety by a finite abelian group.
  \item\label{introitem:counterexample} Moreover, no open neighborhood of a singular point of $X$ is a quotient of a smooth variety by a finite abelian group.
  \end{enumerate}
\end{theorem}

\begin{remark}
  The property of being a quotient of a smooth variety by a finite abelian group is prima facie a global property. Question \ref{q:abelian} asks if this property is in fact \'etale local, and Theorem \ref{thm:P2A5} shows that it is not. However, the equivalence of conditions (\ref{introitem:finite-quot}) and (\ref{introitem:divisors}) of Theorem \ref{thm:main-intro} shows that it is Zariski local.
\end{remark}

There are other variants of Question \ref{q:main} in the literature whose answers are known to be positive. If one modifies Question \ref{q:main} by dropping the assumption that $G$ be finite, then the answer is ``yes'': by \cite[Corollary 2.20]{brauer}, if $X$ is a variety with quotient singularities over a field of characteristic 0, then $X=U/G$, where $U$ is a smooth scheme and $G$ is a linear algebraic group.

If one modifies Question \ref{q:main} in a different direction, requiring a finite surjection $U\to X$ with $U$ smooth, but no group action, then the answer is also ``yes'': it follows from \cite[Theorem 1]{kv} and \cite[Theorem 2.18]{brauer} that for an irreducible quasi-projective variety $X$ with quotient singularities over a field $k$, there is a finite surjection from a smooth variety to $X$.

Question \ref{q:main} therefore asks if there is a common refinement of these two variants.

\begin{remark}[What we expect a full answer to Question \ref{q:main} to look like]
Although Question \ref{q:main} is scheme-theoretic, stacks are a natural tool for answering it. The question is equivalent to asking if there is a smooth Deligne-Mumford stack $\X$ which is a quotient stack by a finite group, and which has coarse space $X$.

By \S\ref{subsec:bg-representability}, expressing a stack $\X$ as a quotient by group $G$ (i.e.~finding a representable map to $BG$) is a matter of finding a vector bundle $\V$ on $\X$ where the stabilizers of $\X$ act faithfully on the fibers of $\V$, and where $\V$ has some additional property which ensures that the induced map $\X\to BGL(\rk(\V))$ factors through $BG$ for some inclusion $G\subseteq GL(\rk(\V))$. For $G$ a torus, this condition is that $\V$ is a sum of line bundles. For $G$ finite diagonalizable, this condition is that $\V$ is a sum of torsion line bundles.

In \cite[Theorem 1]{bottomupDM}, it is shown that choosing a smooth stack with a given coarse space $X$ is a matter of choosing the ramification divisor $D$ of the coarse space map (after which $\X\cong \sqrt{D/X^\can}^\can$). With a sufficient understanding of what vector bundles are introduced by canonical stack and root stack constructions, expressing $X$ as a quotient scheme is a matter of choosing divisors on $X$ which introduce a vector bundle with the properties in the previous paragraph, and showing that $X$ cannot be expressed as a quotient is a matter of showing that there is no such choice of divisors.
\hfill$\diamond$
\end{remark}

\subsection*{Acknowledgments}
We thank Dan Abramovich, Dan Edidin, Tom Graber, and David Rydh for helpful conversations.  We especially thank Bill Fulton for suggesting Question \ref{q:main} and for the interest he took in this project.

\section{Answering Question \ref{q:main} affirmatively for torus quotients}
\label{sec:main}
This section is devoted to the proof of Theorem \ref{thm:main2}, which shows that (\ref{introitem:torus-quot}) implies (\ref{introitem:finite-quot}) in Theorem \ref{thm:main-intro} (see Remark \ref{rmk:thmsec2=>thmintro}). We begin by collecting some well-known definitions and results.

\bigskip

Given a group scheme $G$ over a field $k$, its \emph{Cartier dual} $D(G)$ is defined to be the group scheme of characters $\hom(G,\GG_m)$. A group scheme is called \emph{diagonalizable} if it is of the form $D(A)$ for a finitely generated abelian group $A$. A group scheme is called \emph{tame} if its order is prime to the characteristic of $k$. Note that diagonalizable group schemes are simply group schemes of the form $\GG_m^r\times\mu_{n_1}\times\dots\times\mu_{n_\ell}$. In particular, a tame finite diagonalizable group scheme over an algebraically closed field is simply a finite abelian constant group.


A scheme $X$ is said to have \emph{(tame) diagonalizable quotient singularities} if it is \'etale locally the quotient of a smooth scheme by a (tame) finite diagonalizable group scheme.

\subsection{Background on quotients}
\label{subsec:bg-quotients}

Throughout this paper, we will denote by $U/G$ the coarse space of the stack $[U/G]$, where $U$ is a finitely presented algebraic space and $G$ is an affine group scheme acting properly\footnote{The assumption that $G$ acts \emph{properly} is imposed to guarantee the existence of a coarse space with good properties. Properness of the action is equivalent to the diagonal of the stack $[U/G]$ being proper. Since the stabilizers of points are assumed finite, $\Delta_{[U/G]}$ is proper and quasi-finite, so it is finite. Thus, $[U/G]$ is locally finitely presented and has finite inertia, so the Keel-Mori Theorem \cite[Theorem 1.1]{conrad-KM} applies.} on $U$ with tame finite stabilizers.
When we work with quasi-projective varieties and geometric quotients (in the sense of \cite{git}), there are several technical questions one may worry about: do these two notions of quotient agree? If $U$ is a quasi-projective variety, must $U/G$ be a quasi-projective variety (and vice versa)? If the geometric quotient of $U$ by $G$ exists, is the action automatically proper? The following lemma resolves these concerns for all cases we will consider.
\begin{lemma}\label{lem:quotient=coarse-space}{\ }
 \begin{enumerate}
  \item\label{item:finite->proper} Any action of a finite group $G$ on a separated algebraic space $U$ is proper.
  \item\label{item:finite->q-proj} If a finite group $G$ acts on a quasi-projective variety $U$, then the geometric quotient of $U$ by $G$ exists, is quasi-projective, and agrees with the coarse space of the stack quotient $[U/G]$.
  \item\label{item:q-proj->quotients-agree} Suppose a torus $T$ acts on a quasi-projective variety $V$ with tame finite stabilizers. If the geometric quotient $X$ exists and is a quasi-projective variety, then the action of $T$ on $V$ is proper and $X$ is the coarse space of the stack quotient $[V/T]$.
  \item\label{item:algsp->q-proj-var} Suppose an affine group $G$ acts properly on an algebraic space $U$, with tame finite stabilizers. If the coarse space $X$ of $[U/G]$ is a quasi-projective variety, then $U$ is a quasi-projective variety.
 \end{enumerate}
\end{lemma}
\begin{proof}
 \ref{item:finite->proper}. The action is proper if and only if the action map $G\times U\to U\times U$ given by $(g,u)\mapsto (u,g\cdot u)$ is proper. Since $G$ is finite, the two maps $(g,u)\mapsto u$ and $(g,u)\mapsto g\cdot u$ are proper. Since $G$ and $U$ are separated, $G\times U$ is separated. The action map is the composition $G\times U\to (G\times U)\times (G\times U)\to U\times U$ of a closed immersion and a product of two proper maps.

 \ref{item:finite->q-proj}. The geometric quotient of $U$ by $G$ exists and is quasi-projective by \cite[Chapter IV, Proposition 1.5]{knutson}. Every point of $U$ is contained in a $G$-invariant open affine.\footnote{Say $U\subseteq \PP^n$ is locally closed. For any point $u\in U$, by graded prime avoidance, there is a positive degree homogeneous element $f$ in the ideal of $\overline U\setminus U$ which does not vanish on any point in the orbit $G\cdot u$. The non-vanishing locus of $f$, $D(f)$, is an affine neighborhood of $G\cdot u$, so $\bigcap_{g\in G}g\cdot D(f)$ is a $G$-invariant neighborhood of $G\cdot u$, and is again affine since $U$ is separated.}
If $\spec A$ is a $G$-invariant open affine neighborhood of $u\in U$, both the geometric quotient of $\spec A$ by $G$ and the coarse space of $[(\spec A)/G]$ are given by $\spec A^G$. Geometric quotients and coarse spaces can both be constructed locally and then glued, so the geometric quotient of $U$ by $G$ agrees with the coarse space of $[U/G]$.

 \ref{item:q-proj->quotients-agree}. By \cite[Converse 1.13]{git}, the action of $T$ on $V$ is proper. By \cite[Corollary 2]{sumihiro}, every point of $V$ has an invariant open affine neighborhood, so the quotient morphism $V\to X$ is affine. For an invariant open affine $\spec A \subseteq V$, both the geometric quotient of $\spec A$ by $T$ and the coarse space of $[(\spec A)/T]$ are given by $\spec A^T$. Geometric quotients and coarse spaces can both be constructed locally and then glued, so the geometric quotient of $V$ by $T$ agrees with the coarse space of $[V/T]$.

 \ref{item:algsp->q-proj-var}. Since $U\to [U/G]$ is a $G$-torsor (so affine) and $[U/G]\to X$ is a coarse space morphism (so is cohomologically affine), we have that $U\to X$ is affine. Thus, if $X$ is a quasi-projective variety, so is $U$.
\end{proof}

\subsection{Background on representability}
\label{subsec:bg-representability}

\begin{lemma}[Criteria for Representability]\label{lem:representability-criteria}
 A morphism of Artin stacks $f\colon \X\to \Y$ is representable if and only if the following equivalent conditions hold:
 \begin{enumerate}
  \item\label{criterion:geometric-fibers} the geometric fibers of $f$ are algebraic spaces;
  \item\label{criterion:injective-stabilizers} for every geometric point $x$ of $\X$, the induced map of stabilizer groups $\stab_\X(x)\to \stab_\Y(f(x))$ is injective.
 \end{enumerate}
\end{lemma}
\begin{proof}
The equivalence of (\ref{criterion:geometric-fibers}) with representability of $f$ is shown in \cite[Corollary 2.2.7]{conrad:elliptic}. To show the equivalence of (\ref{criterion:geometric-fibers}) and (\ref{criterion:injective-stabilizers}), it suffices to base change by the map $\spec k\to\Y$ determined by $f(x)$. We can therefore assume $\Y$ is the spectrum of an algebraically closed field. In this case, we must show that $\X$ is an algebraic space if and only if its stabilizers at all geometric points are trivial.  This is shown in \cite[Theorem 2.2.5(1)]{conrad:elliptic}.
\end{proof}

\begin{corollary}\label{cor:representability}
 Suppose $\X$, $\Y$, and $\Z$ are algebraic stacks, and $f\colon \X\to \Y$ is a morphism.
 \begin{enumerate}
  \item\label{rep:property-P} If $g\colon \Y\to \Z$ is a representable morphism, then $f$ is representable if and only if $g\circ f$ is representable.
  \item\label{rep:section} If $f$ has a section, it is representable.
  \item\label{rep:product} If $f$ is representable and $g\colon\X\to \Z$ is any morphism, $(f,g)\colon \X\to \Y\times\Z$ is representable.
  \item\label{rep:local} If $\Z\to \Y$ is a surjective locally finite type morphism, then $f$ is representable if and only if $f_\Z:\X\times_\Y \Z\to \Z$ is representable.
 \end{enumerate}
\end{corollary}
\begin{proof}
 (\ref{rep:property-P}), (\ref{rep:section}), and (\ref{rep:product}) follow immediately from Lemma \ref{lem:representability-criteria}(\ref{criterion:injective-stabilizers}). (\ref{rep:local}) follows from Lemma \ref{lem:representability-criteria}(\ref{criterion:geometric-fibers}), as the geometric fibers of $f$ and $f_\Z$ are the same.
\end{proof}

\begin{lemma}\label{lem:quot<->representable}
 Suppose $\X$ is an algebraic stack and $G$ is a locally finite type group scheme over an algebraic space $S$. If $U\to \X$ is a $G$-torsor, then the corresponding morphism $\X\to BG$ is representable if and only if $U$ is an algebraic space.
\end{lemma}
\begin{proof}
 The following diagram is cartesian.
 \[\xymatrix{
  U\ar[r]\ar[d] & S\ar[d]\\
  \X\ar[r] & BG
 }\]
 By Corollary  \ref{cor:representability}(\ref{rep:local}), $\X\to BG$ is representable if and only if $U\to S$ is representable (i.e.~if and only if $U$ is an algebraic space).
\end{proof}


The following proposition is a variant of \cite[Lemma 2.12]{brauer}.
\begin{proposition}\label{prop:quot<->bundle}
 Let $\X$ be an algebraic stack.  Then $\X\cong [V/G]$ for some algebraic space $V$ and subgroup $G\subseteq GL_r$ if and only if there is a rank $r$ vector bundle $\E$ on $\X$ such that the stabilizers at geometric points of $\X$ act faithfully on the fibers of $\E$.  Moreover, $G$ can be taken to be $\GG_m^r$ exactly when $\E$ can be taken to be the direct sum of line bundles; and $G$ can be taken to be a finite diagonalizable group scheme exactly when $\E$ can be taken to be the direct sum of \emph{torsion} line bundles.
\end{proposition}

\begin{remark}
  The proof of Proposition \ref{prop:quot<->bundle} shows that the vector bundle $\E$ is the pullback of a universal bundle on $BG$ along the morphism $\X\to BG$ corresponding to the $G$-torsor $V\to \X$.
\end{remark}

%
%
\begin{proof}
 By Lemma \ref{lem:quot<->representable}, an isomorphism $\X\cong [V/G]$ where $V$ is an algebraic space is equivalent to a representable morphism $f\colon \X\to BG$. Given a morphism $\X\to BG$, let $\E$ be the pullback of the universal rank $r$ vector bundle on $BGL_r$ for $G\subseteq GL_r$ (resp.~the universal sum of line bundles on $B\GG_m^r$ for $G=\GG_m^r$, resp.~the universal sum of torsion line bundles on $BG$ for $G\subseteq \GG_m^r$ finite diagonalizable). For a geometric point $x$ of $\X$, the action of $\stab_\X(x)$ on the fiber of $\E$ at $x$ is given by the morphism of stabilizers induced by $\X\to BG$. By Lemma \ref{lem:representability-criteria}(\ref{criterion:injective-stabilizers}), the stabilizer action on the fibers is faithful at all geometric points if and only if $\X\to BG$ is representable.
\end{proof}


\subsection{Background on residual separability}
\label{subsec:bg-residual-separatedness}

For Bertini arguments in positive characteristic, it is important to restrict to the following class of morphisms.
\begin{definition}\label{def:res-sep}
A morphism $f:X\to Y$ of schemes is \emph{residually separable} if for all $x\in X$, the induced extension of residue fields $k(x)/k(f(x))$ is separable (i.e.~$\spec k(x)\to \spec {k(f(x))}$ is geometrically reduced). A line bundle $\L$ on $X$ is called \emph{residually separable} if the map to projective space induced by some finite-dimensional base-point free linear system of $\L$ is residually separable.
A morphism of Deligne-Mumford stacks $\X\to\Y$ is \emph{residually separable} if there exist \'etale covers $U\to\X$ and $V\to\Y$ with $U$ and $V$ schemes, and a residually separable morphism $f:U\to V$ over $\X\to\Y$.
\end{definition}
The following lemma shows that the definition of residually separable morphisms of Deligne-Mumford stacks is compatible with the definition for schemes.
\begin{lemma}
\label{l:res-sep}
Let $f\colon X\to Y$, $g\colon Y\to Z$, and $h\colon Z'\to Z$ be morphisms of schemes.
\begin{enumerate}
  \item\label{res-sep:comp} If $f$ and $g$ are residually separable, then so is $g\circ f$.
  \item\label{res-sep:source} If $f$ is an \'etale cover, then $g$ is residually separable if and only if $g\circ f$ is.
  \item\label{res-sep:local} If $h$ is an \'etale cover, then $g$ is residually separable if and only if its pullback $g'\colon Y\times_Z Z'\to Z'$ is residually separable.
\end{enumerate}
\end{lemma}
\begin{proof}
  By \cite[Lemma \href{http://math.columbia.edu/algebraic_geometry/stacks-git/locate.php?tag=035Z}{035Z} (2)]{stacks-project}, a scheme $W$ over a field $k$ is geometrically reduced if and only if $W\times_k V$ is reduced for all reduced $k$-schemes $V$. From this, (\ref{res-sep:comp}) follows easily. Since reducedness may be checked \'etale locally, we have (\ref{res-sep:source}).
  
  For (\ref{res-sep:local}), let $y'$ denote a point of $Y\times_Z Z'$, and let $y$, $z'$, and $z$ denote its images in $Y$, $Z'$ and $Z$, respectively. For a geometric point $p$ of $z'$, the map $p\times_{z'}y'\to p\times_z y$ is \'etale since it is a composition of an open immersion $p\times_{z'}y'\to p\times_zy'$ (a pullback of the diagonal $z'\to z'\times_zz'$) and an \'etale map $p\times_zy'\to p\times_zy$ (a pullback of the \'etale map $y'\to y$). If $g$ is residually separable, then $p\times_zy$ is reduced, so $p\times_{z'}y'$ is reduced, showing that $g'$ is residually separable. Conversely, suppose $g'$ is residually separable and $p$ is a geometric point of $z$. Then $p$ factors through $z'$, and $p\times_{z'}y'$ is reduced. Since reducedness is \'etale local, $p\times_z y$ is reduced, showing that $g$ is residually separable.
\end{proof}
\begin{corollary}\label{cor:res-sep}
If $\X$ is a tame, separated, locally finitely presented Deligne-Mumford stack, then its coarse space map $\X\to X$ is residually separable.
\end{corollary}
\begin{proof}
We can check that $\X\to X$ is residually separable by base changing to the points of $X$, so we may assume $X=\spec k$, where $k$ is a field. Let $x\in\X$ be the unique point. By \cite[Lemma 2.2.3]{Abramovich-Vistoli}, after replacing $k$ by an \'etale extension, we may assume $\X = [U/G]$, where $G$ is the stabilizer group of $x$ and $U=\spec A$ is a connected affine scheme. Then it suffices to check that $U\to \spec k$ is residually separable. This may be checked after replacing $U$ by its reduction, in which case $A$ is a reduced connected finite extension of $k$. Hence, $A$ is a field and $k$ is the subfield of invariants under the action of $G$, so the degree of the extension must divide the order of $G$. Since $\X$ is tame, this order is relatively prime to the characteristic, so the extension is separable.
\end{proof}

\subsection{Proof of (\ref{introitem:torus-quot})\texorpdfstring{$\Rightarrow$}{=>}(\ref{introitem:finite-quot}) in Theorem \ref{thm:main-intro}}
\label{subsec:proof-of-2=>1}

With the above preliminaries in place, we now turn to the main theorem of this section.
\begin{theorem}\label{thm:main2}
 Let $k$ be an infinite field and let $X=V/H$ be a separated algebraic space, where $V$ is a quasi-compact smooth algebraic space and $H$ is a diagonalizable group scheme over $k$ which acts properly on $V$ with tame finite stabilizers.  Further assume that every coset of $n\cdot\pic(X)\subseteq\pic(X)$ contains a residually separable base-point free line bundle, where $n$ is the least common multiple of the exponents of the stabilizers of the $H$-action.\footnote{To see $n$ is finite, note that $\X:=[V/H]$ is separated Deligne-Mumford. Therefore, there is an \'etale cover $\{X_i\to X\}_i$ such that $\X\times_X X_i=[U_i/G_i]$ for an algebraic space $U_i$ and finite group $G_i$ \cite[Lemma 2.2.3]{Abramovich-Vistoli}. The exponents of the stabilizers of geometric points of $\X_i$ divide the exponents of $G_i$.  Since $\X$ is quasi-compact, we can take the cover of $X$ to be finite.}
 Then $X=U/G$, where $U$ is a smooth algebraic space and $G$ is a \emph{finite} diagonalizable group scheme over $k$ acting properly on $U$.
\end{theorem}
\begin{remark}\label{rmk:v.ample=>res-sep}
If $X$ is quasi-projective, then every coset of $n\cdot\pic(X)\subseteq\pic(X)$ automatically contains a residually separable base-point free line bundle.  Indeed, every line bundle $\L$ on $X$ can be made very ample, and hence base-point free, after twisting by a sufficiently high power of $\O_X(n)\in n\cdot\pic(X)$. Since the induced map from $X$ to projective space is an immersion, the residue field extensions are trivial, so very ample line bundles are residually separable.
\end{remark}
\begin{remark}\label{rmk:thmsec2=>thmintro}
  Theorem \ref{thm:main2}, Remark \ref{rmk:v.ample=>res-sep}, and Lemma \ref{lem:quotient=coarse-space}(\ref{item:q-proj->quotients-agree} and \ref{item:algsp->q-proj-var}) show that (\ref{introitem:torus-quot}) implies (\ref{introitem:finite-quot}) in Theorem \ref{thm:main-intro}.
\end{remark}
\begin{proof}[\ifnotlms Proof \fi of Theorem \ref{thm:main2}]
 By Proposition \ref{prop:quot<->bundle}, to show $X$ is a global quotient of a smooth algebraic space by a finite diagonalizable group, it suffices to find a smooth Deligne-Mumford stack $\Y$ with coarse space $X$, and torsion line bundles on $\Y$ such that the stabilizers of $\Y$ act faithfully on the fibers of the direct sum of the line bundles.

 By hypothesis, $X=V/H$ where $V$ is a smooth algebraic space and $H$ is a diagonalizable group scheme. Properness of the action of $H$ on $V$ is equivalent to the condition that the stack quotient $\X=[V/H]$ is separated. By Proposition \ref{prop:quot<->bundle}, there are line bundles $\M_1,\dots,\M_r$ on $\X$ so that the stabilizers of $\X$ at geometric points act faithfully on the fibers of $\M_1\oplus\dots\oplus\M_r$. The stabilizers act trivially on $\M_i^{\otimes n}$, so by \cite[Thm 10.3]{Alper:good} there is a line bundle $\K_i$ on $X$ such that
 \[
  \M_i^{\otimes n}=\phi^*\K_i,
 \]
 where $\phi:\X\to X$ is the coarse space map.
 \begin{remark}
  Instead of choosing a single $n$, we may choose $n_i$ such that $\M_i^{\otimes n_i}$ has trivial residual representations. This is desirable if one is explicitly constructing $U$ and $G$ for a specific $X$, but to make the proof more readable, we use a single $n$.
 \hfill$\diamond$
 \end{remark}
 By hypothesis, for each $i$ there exists $\N_i\in\pic(X)$ such that $\K_i\otimes\N_i^{\otimes n}$ is base-point free and residually separable.  Since the stabilizers of $\X$ act trivially on the fibers of $\phi^*\N_i$, the residual representations of $\M_i\otimes \phi^*\N_i$ are isomorphic to the residual representations of $\M_i$. Replacing $\M_i$ by $\M_i\otimes \phi^*\N_i$, we may therefore assume that each $\K_i$ is base-point free and residually separable.

 Let $\Phi:W\to\X$ be an \'etale cover by a scheme and let $\psi_i:W\to\PP^{d_i}$ denote the composite of $\Phi$, the coarse space map $\phi:\X\to X$, and the map defined by a base-point free, residually separable linear system of $\K_i$.  Since $\Phi$ is \'etale and representable, it is residually separable.  Since $\X$ is a tame Deligne-Mumford stack, Corollary \ref{cor:res-sep} shows that $\phi$ is residually separable. It follows from Lemma \ref{l:res-sep}(\ref{res-sep:comp}) that $\psi_i$ is residually separable as well.

 Applying \cite[Corollary 4.3]{bertini2} to $\psi_i$, a generic section of $\psi_i^*\O(1)=\Phi^*\phi^*\K_i=\Phi^*(\M_i)^{\otimes n}$ has smooth vanishing locus. Let $Z\subseteq V$ be the closed locus where $H$ does not act freely, let $\Z=[Z/H]\subseteq \X$, and let $Z'=\Z\times_{\X}W\subseteq W$.  We may then choose sections $s_{i,1},\dots,s_{i,c_i}$ of $\psi_i^*\O(1)$ for each $i$ satisfying the following properties: the Cartier divisor $D_{i,j}$ defined by the vanishing of $s_{i,j}$ is smooth for each $i$ and $j$, the set of divisors $\{D_{i,j}\}_{i,j}$ has simple normal crossings, and $Z'\cap\bigcap_j D_{i,j}=\varnothing$ for each $i$. Since the $D_{i,j}$ are obtained as pullbacks of hyperplane sections of $X$, they descend to smooth Cartier divisors $\D_{i,j}$ on $\X$ with simple normal crossings such that for each $i$, $\Z\cap \bigcap_j\D_{i,j}=\varnothing$. Let $\pi:\Y\to\X$ denote the $n$-th root construction along each of the $\D_{i,j}$ (see for example \cite[1.3.b]{fmn}).
 Since the $\D_{i,j}$ have simple normal crossings, $\Y$ is smooth by \cite[1.3.b.(3)]{fmn}.\footnote{We choose the $\D_{i,j}$ to have simple normal crossings so that $\Y$ is smooth, but it is possible to build $\Y$ by rooting one divisor at a time, in which case smoothness of that divisor is sufficient to ensure every step is smooth. Here is a sketch of the argument. Applying \cite[Corollary 4.3]{bertini2}, we may choose a section $s_{1,1}$ of $\Phi^*\phi^*\K_1$ so that its vanishing divisor $D_{1,1}$ is smooth and intersects $Z$ properly (for a generic choice of $s_{1,1}$, $D_{1,1}$ will be smooth, and avoiding a given point on each connected component of $Z$ is an open condition). This $D_{1,1}$ descends to a smooth Cartier divisor $\D_{1,1}$ on $\X$. Let $\pi\colon \X'\to\X$ be the $n$-th root stack of $\X$ along $\D_{1,1}$. Since $\X$ and $\D_{1,1}$ are smooth, $\X'$ is smooth by \cite[1.3.b.(3)]{fmn}.
 One shows that this $\X'$ is a quotient of some smooth scheme $V'$ by a torus using the same argument that is used for $\Y$ in the remainder of the proof of Theorem \ref{thm:main2}. Now we replace $\X$, $V$, and $Z$ by $\X'$, $V'$, and the preimage of $[(Z\cap D_{1,1})/H]$ in $V'$ and repeat the argument, producing $D_{1,j}$, until $Z$ is empty. Then we replace $Z$ by the preimage of the \emph{original $Z$} and repeat the argument for $\K_2$, $\K_3$, etc.}

 Let $\M'_{i,j}$ denote the universal line bundles on the root stack $\Y$ for which
 \[
 (\M'_{i,j})^{\otimes n}=\pi^*\M_i^{\otimes n},
 \]
 and let
 \[
  \L_{i,j}:=\M'_{i,j}\otimes\pi^*\M_i^\vee.
 \]
 We show that the stabilizers of $\Y$ act faithfully on the fibers of $\bigoplus_{i,j}\L_{i,j}$.  Since the $\L_{i,j}$ are torsion, Proposition \ref{prop:quot<->bundle} will show that $\Y$ is a quotient of a smooth algebraic space by a finite diagonalizable group scheme. Moreover, since $\Y$ has separated coarse space $X$, it is separated, so the group action is proper. To see that the stabilizer actions are faithful, we show that the corresponding map $\Y\to B\GG_m^C$ is representable, where $C=\sum_i c_i$ is the total number of $s_{i,j}$.  Consider the following diagram in which the square is cartesian.
 \[\xymatrix{
  & \Y\ar[r]^-g\ar[d]_\pi  & [\AA^1/\GG_m]^C\ar[d]^{\hat{}n}\ar[r]^-u & B\GG_m^C\\
  B\GG_m^r & \X \ar[r]^-{f}\ar[l]_-m & [\AA^1/\GG_m]^C
 }\]
 The morphism $m$ corresponds to the $r$-tuple of line bundles $(\M_i)_i$, $u\circ g$ corresponds to $(\M'_{i,j})_{i,j}$, $f$ corresponds to the $(\D_{i,j})_{i,j}$, $u$ is the forgetful map which sends $C$ line bundles with sections to the underlying line bundles, and $\hat{}n$ sends $C$ line bundles with sections $(\E_i,t_i)$ to their $n$-th tensor powers $(\E^{\otimes n}_i,t^{\otimes n}_i)$.

 Proposition \ref{prop:quot<->bundle} shows that $m$ is representable. Lemma \ref{lem:representability-criteria}(\ref{criterion:geometric-fibers}) shows that $u$ is representable, as the pullback of $u$ under the universal $\GG_m^C$-torsor is $(\AA^1)^C$.  From the following cartesian diagram and the fact that the diagonal map $\Delta$ is representable, we see that $(\pi,g)$ is representable.
 \[\xymatrix{
 \Y\ar[r]^-{(\pi,g)}\ar[d] & \X\times [\AA^1/\GG_m]^C\ar[d]^{f\times\;\hat{}n}\\
 [\AA^1/\GG_m]^C\ar[r]^-{\Delta} & [\AA^1/\GG_m]^C\times[\AA^1/\GG_m]^C
 }\]
The morphism $(m\circ\pi,u\circ g)\colon \Y\to B\GG_m^r\times B\GG_m^C$ corresponding to $\bigl((\pi^*\M_i)_i, (\M'_{i,j})_{i,j}\bigr)$ is representable, as it is the composition of representable morphisms $(m,\id)\circ(\id,u)\circ(\pi,g)$.

 To show that $\ell\colon \Y\to B\GG_m^C$ corresponding to $(\L_{i,j})_{i,j}$ is representable, we may restrict to the residual gerbe of a geometric point $y$ of $\Y$ by Lemma \ref{lem:representability-criteria}(\ref{criterion:injective-stabilizers}). By Corollary \ref{cor:representability}(\ref{rep:property-P}), it then suffices to find a representable morphism $h\colon B\GG_m^C\to B\GG_m^{r+C}$ such that the composite $h\circ \ell$ agrees with $(m\circ \pi,u\circ g)$ at $y$ (i.e.~when restricted to the residual gerbe at $y$).

 Let $\widetilde\Z=\pi^{-1}(\Z)$. Suppose $y$ is a geometric point of $\Y$ in the complement of $\widetilde\Z$. The residual representations of the $\M_i$ are trivial on the complement of $\Z$, so the residual representations of the $\pi^*\M_i$ are trivial on the complement of $\widetilde\Z$, so $\L_{i,j}\cong \M_{i,j}'$ at $y$. Letting $h\colon B\GG_m^C\to B\GG_m^{r+C}$ be the morphism sending a $C$-tuple of line bundles $(\E_1,\dots, \E_C)$ to $(\O,\dots, \O,\E_1,\dots, \E_C)$, we see that $h$ is representable by Corollary \ref{cor:representability}(\ref{rep:section}), and that $h\circ \ell=(m\circ\pi,u\circ g)$ at $y$.

 Now suppose that $y$ is a geometric point of $\widetilde\Z$. For each $i$, there exists some $j_i$ such that $y$ is not contained in $\D_{i,j_i}$. We then have that the residual representation of $\M_{i,j_i}$ at $y$ is trivial, so $\pi^*\M_i \cong \L_{i,j_i}^\vee$ at $y$. Let $h_1\colon B\GG_m^C\to B\GG_m^r\times B\GG_m^C$ be the morphism given by sending the $C$-tuple $(\L_{i,j})_{i,j}$ to $\bigl((\L_{i,j_i}^\vee)_i,,(\L_{i,j})_{i,j}\bigr)$. Let $h_2\colon B\GG_m^r\times B\GG_m^C\to B\GG_m^r\times B\GG_m^C$ be given by sending $\bigl((\M_i)_i,(\L_{i,j})_{i,j}\bigr)$ to $\bigl((\M_i)_i,(\L_{i,j}\otimes \M_i)_{i,j}\bigr)$. Then $h_1$ and $h_2$ are representable by Corollary \ref{cor:representability}(\ref{rep:section}), so $h=h_2\circ h_1$ is representable. We have that $h\circ \ell=(m\circ\pi, u\circ g)$ at $y$, so $\ell$ is representable at $y$. 
\end{proof}

As a consequence, we see that Question \ref{q:main} has an affirmative answer for simplicial toric varieties with tame singularities.
\begin{corollary}
\label{cor:main}
Every quasi-projective toric variety with tame quotient singularities over an algebraically closed field $k$ is of the form $U/G$, where $U$ is a smooth $k$-variety and $G$ is a finite abelian group.
\end{corollary}
\begin{proof}
  Since $X$ is a toric variety with quotient singularities, the Cox construction shows $X=V/H$, where $V$ is smooth and $H$ is a diagonalizable group scheme acting on $V$ with finite stabilizers (see \S\ref{subsec:lb} for a review of the Cox construction), so Theorem \ref{thm:main2} applies.
\end{proof}

\section{Explicit construction for toric varieties}\label{sec:explicit}

To emphasize that our proof of Theorem \ref{thm:main2} is constructive, we reinterpret it in the special case when $X$ is a toric variety. The end result is the procedure described in the following theorem, which we demonstrate in \S\ref{sec:ex}. We urge the reader to refer to \S\ref{sec:ex} immediately, as it clarifies the construction of $\hhat\Sigma$ and shows the simplicity of the procedure in practice.

\begin{theorem}\label{thm:tv}
 Let $X$ be a quasi-projective toric variety with tame quotient singularities over an infinite  field $k$. Let $\Sigma$ be the fan of $X$, let $Z\subseteq X$ be the singular locus, and let $X=V/H$ be the Cox construction of $X$ (see \S\ref{subsec:lb}).
 \begin{enumerate}
  \item[1.] There exist Weil divisors $D_1,\dots, D_r$ which generate the class groups of all torus-invariant open affine subvarieties of $X$. Letting $n_i$ be integers so that $n_i D_i$ is Cartier for each $i$, the $D_i$ can be chosen so that $n_i  D_i$ is very ample.
  \item[2.] There exist sections $\{s_{i,j}\}_{1\le j\le c_i}$ of $\O_X(n_i D_i)$ so that the preimages of the vanishing loci $\{V(s_{i,j})\}_{i,j}$ in $V$ are smooth and have simple normal crossings, and for each $i$, $\bigcap_j V(s_{i,j})$ is disjoint from $Z$.
  \item[3.] Let $W$ be the toric variety with fan $\hhat\Sigma$, as described in \S\ref{subsec:total-spaces}, and let $U_{i,j}\subseteq W$ be the $s_{i,j}$-cut together with its $\mu_{n_i}$-action (see Definition \ref{def:tv-cut}). Then the scheme-theoretic intersection $U=\bigcap_{i,j} U_{i,j}$ in $W$ is a smooth variety with an action of $G=\prod_i \mu_{n_i}^{c_i}$, such that $X\cong U/G$.
 \end{enumerate}
\end{theorem}

In \S\ref{subsec:lb} we review the Cox construction, which expresses every toric variety $X$ as a quotient $V/H$, where $V$ is smooth and $H$ is a diagonalizable group scheme.  We also give a non-stacky description of how to find line bundles $\M_1,\dots,\M_r$ on $\X=[V/H]$ such that the stabilizers of $\X$ act faithfully on the fibers of $\bigoplus\M_i$.  Finding such line bundles is the starting point in the proof of Theorem \ref{thm:main2}. In \S\ref{subsec:pedroot}, we give an explicit description of $n$-th root constructions of $n$-th tensor powers of line bundles.  In \S\ref{subsec:total-spaces}, we describe the fan $\hhat \Sigma$ which appears in Theorem \ref{thm:tv}.  Lastly, in \S\ref{subsec:tv} we put these results together to show how the proof of Theorem \ref{thm:main2} yields the procedure described in Theorem \ref{thm:tv}.

\subsection{Canonical stacks, line bundles, and divisors}
\label{subsec:lb}

\subsubsection*{Canonical Stacks}

If $X$ is a finite-type algebraic space, then \cite[2.9 and the proof of 2.8]{vistoli:intersection} shows $X$ has tame quotient singularities if and only if $X$ is the coarse space of a smooth tame Deligne-Mumford stack. Moreover, in this case there is a smooth tame Deligne-Mumford stack $X^\can$ with coarse space $X$ such that the coarse space morphism $X^\can\to X$ is an isomorphism away from codimension 2. We refer to $X^\can$ as \emph{the canonical stack} over $X$.

\begin{remark}[Universal property of canonical stacks {\cite[Theorem 4.6]{fmn}}]\label{rmk:canonical-universal-property}
  Suppose $\X$ is a smooth Deligne-Mumford stack with coarse space morphism $\pi\colon\X\to X$ which is an isomorphism away from a locus of codimension at least 2. Then $\X$ is universal (terminal) among smooth Deligne-Mumford stacks with trivial generic stabilizer and a dominant codimension-preserving morphism to $X$.
  
This universal property is stable under base change by \'etale morphisms (or any other codimension-preserving morphisms). That is, for $U\to X$ \'etale, $X^\can\times_X U$ is the canonical stack over $U$.
\end{remark}

\begin{remark}[Local structure of canonical stacks]\label{rmk:canonical-local-rings}
  Suppose $x\in X$ is a tame quotient singularity, and $U/G\to X$ is an \'etale morphism with $x$ in its image, so that $U$ is smooth and $G$ is a finite group. Let $u\in U$ be a preimage of $x$; we may assume $G$ fixes $u$, replacing $G$ by the stabilizer of $u$ if necessary. Let $P\subseteq G$ be the subgroup generated by \emph{pseudoreflections} through $u$, i.e.~subgroups which fix a divisor through $u$. The construction in \cite{vistoli:intersection} shows that $U/P$ is smooth (possibly only after replacing $U$ by an open neighborhood of $u$) and $G/P$ acts on $U/P$ without pseudoreflections. Therefore, in an \'etale neighborhood of $x$, the canonical stack of $X$ is given by $[(U/P)/(G/P)]$.

  On the level of formal local rings, we have $\hhat\O_{X,x}=k[[V]]^G$, where $V$ is a vector space (the tangent space at $u$) and $G$ is a tame finite group acting linearly on $V$. Then $k[[V]]^P\cong k[[W]]$ for a vector space $W$, $G/P$ acts on $W$ without pseudoreflections, and $X^\can\times_X\spec\hhat\O_{X,x}\cong [\spec k[[V]]^P/(G/P)]$.
\end{remark}

\begin{remark}[$\Cl(X)\cong \pic(X^\can)$]\label{rmk:canonical-pic=cl}
  Let $U\subseteq X$ be the smooth locus of $X$. Then $\pi\colon X^\can\to X$ is an isomorphism over $U$. Since the complement of $U$ in $X$ and the complement of the preimage of $U$ in $X^\can$ are of  codimension at least 2, we have a chain of isomorphisms
  \[
   \Cl(X)\cong \Cl(U)\cong \Cl(X^\can)\cong \pic(X^\can)
  \]
  where the last isomorphism follows from the fact that every Weil divisor on $X^\can$ is Cartier (one may check \'etale locally that an ideal sheaf is a line bundle, and $X^\can$ is smooth). So given a Weil divisor $D$ on $X$, it makes sense to speak of the associated line bundle $\O_{X^\can}(D)$ on $X^\can$.
\end{remark}

\subsubsection*{The Cox construction and characterization of jointly faithful line bundles}
Given a normal toric variety $X$ with no torus factors, the Cox construction \cite[\S5.1]{cls} produces an open subscheme $V$ of $\AA^n$ and a subgroup $H\subset\GG_m^n$ such that $X=V/H$. We briefly recall this construction.  If $\Sigma$ is the fan of $X$ and $\Sigma(1)$ denotes its set of rays, then consider the polynomial ring $k[x_\rho\mid\rho\in\Sigma(1)]$ with one variable for each ray. The fan $\Sigma$ for $X$ lives on a lattice $N$ and we have a map $\beta:\ZZ^{\Sigma(1)}\to N$ sending the standard basis vector $e_\rho$ to the first lattice point in $N$ along the ray $\rho$. Letting $(-)^*=\hom_\ZZ(-,\ZZ)$, we have an exact sequence
\[
N^* \xrightarrow{\beta^*} (\ZZ^{\Sigma(1)})^*\to \cok(\beta^*)\to 0
\]
Taking Cartier duals, we have an injection $H:=D(\cok(\beta^*))\to D((\ZZ^{\Sigma(1)})^*) \cong \GG_m^{\Sigma(1)}$. The Cox construction shows that $X=V/H$ where $V\subseteq \AA^{\Sigma(1)}$ is the complement of the closed subscheme defined by the ideal $(\prod_{\rho\notin\sigma}x_\rho\mid\sigma\in\Sigma)$. As an aside, we note that this short exact sequence above shows that $D(H)=\cok(\beta^*) = \Cl(X)$ by \cite[Lemma 5.1.1(a)]{cls}.
From this perspective, the $H$-action on $\AA^{\Sigma(1)}$ can be viewed as coming from the $\Cl(X)$-grading on the polynomial ring $k[x_\rho]$ sending $x_\rho$ to the element in $\Cl(X)$ defined by the divisor $D_\rho$.

If $X$ is a normal toric variety, then there is a unique torus $T'$ and toric variety $X'$ without torus factors such that $X\cong T'\times X'$ (though the isomorphism is non-canonical). Let $X'=V'/H$ be the Cox construction for $X'$, and let $V=T'\times V'$ with trivial $H$-action on the $T'$ factor. We will say $X=V/H$ is the Cox construction for $X$. Since the action of $H$ is free away from a locus of codimension at least 2, $\X=[V/H]=[V'/H]\times T'$ is the canonical stack over $X$ by Remark \ref{rmk:canonical-universal-property}.

\begin{lemma}\label{lem:lb}
 Suppose $D_1,\dots, D_r$ are Weil divisors of a toric variety $X$. The stabilizers of $\X=X^\can$ act faithfully on $\O_\X(D_1)\oplus\dots\oplus\O_\X(D_r)$ if and only if $D_1,\dots, D_r$ generate the Weil class group of every torus-invariant open affine subvariety of $X$.
\end{lemma}
\begin{proof}
 If $U$ is a torus-invariant open affine subvariety of $X$, then $\X\times_X U$ is the canonical stack over $U$ by Remark \ref{rmk:canonical-universal-property}. We may therefore assume $X$ is an affine toric variety. We may prove the lemma after removing a common torus factor from both $X$ and $\X$, so we may further assume $X$ has no torus factor.

 When $X$ is affine with no torus factor, $V=\spec k[x_\rho]$ in the Cox construction described above. In this case, every stabilizer of $\X$ is a subgroup of the stabilizer of the origin in $V$, so we need only show that the stabilizer of the origin, namely $H$, acts faithfully on the fiber of $\bigoplus\O_\X(D_i)$. For a divisor $D$, $H$ acts on the fiber of $\O_\X(D)$ by the character given by the image of $D$ in $\Cl(X)=D(H)$. Therefore, the action of $H$ on the fiber of $\bigoplus\O_\X(D_i)$ has a kernel (factors through a proper quotient) if and only if the images of the $D_i$ in $\Cl(X)$ are contained in a proper subgroup.
\end{proof}
\begin{remark}
 Lemma \ref{lem:lb} implies the following global (but weaker) criterion. If $D_1,\dots, D_r$ generate the quotient of the Weil class group of $X$ by the Cartier class group of $X$, then they generate the Weil class group  of every open affine torus-invariant subvariety $U$ of $X$, and so the stabilizers of $\X$ act faithfully on $\bigoplus \O_\X(D_i)$. To see this, note that the restriction morphism $\Cl(X)\to \Cl(U)$ is surjective, since taking the closure of a Weil divisor on $U$ provides a section. This map sends Cartier divisors to the identity, since Cartier divisors on $U$ are all trivial. We therefore get a surjection $\Cl(X)/\CaCl(X)\to \Cl(U)$.
\end{remark}

\subsection{An alternative description of \texorpdfstring{$n$}{n}-th roots of \texorpdfstring{$n$}{n}-th powers}
\label{subsec:pedroot}

Suppose $X$ is a scheme (or stack), $\L$ is a line bundle on $X$, and $s$ is a global section of $\L^{\otimes n}$.  Informally, we will show that the $n$-th root stack of $\L^{\otimes n}$ along $s$ is the quotient stack $[U/\mu_n]$, where $U\hookrightarrow \VV(\L)$ is the closed subscheme (or closed substack) of $\VV(\L)$ cut out by the equation $t^n=s$, where the $\mu_n$ acts on the ``coordinate'' $t$ with weight 1.

A line bundle $\M$ on $X$ corresponds to a morphism $X\to B\GG_m$. With this correspondence, we have that $\VV(\M) = X\times_{B\GG_m}[\AA^1/\GG_m]$. Sections of $\M$ correspond to sections of the projection $\VV(\M)\to X$, and therefore to maps to $[\AA^1/\GG_m]$ over the given map to $B\GG_m$.

\begin{definition}\label{def:vb-slice}
  For a line bundle $\L$ on $X$ and section $s$ of $\L^{\otimes n}$, we define the \emph{$s$-slice of $\VV(\L)$} to be $X\times_{\VV(\L^{\otimes n})}\VV(\L)$, where $X\to \VV(\L^{\otimes n})$ is given by $s$, and $\hat{}n\colon \VV(\L)\to \VV(\L^{\otimes n})$ is the $n$-th power map. Note that the $\mu_n$-action on $\VV(\L)$ over $\VV(\L^{\otimes n})$ induces a $\mu_n$-action on the $s$-slice over $X$.
\end{definition}

\begin{lemma}
\label{lem:pedroot}
  Let $\L$ be a line bundle on $X$, $s$ a section of $\L^{\otimes n}$, and $U$ the $s$-slice of $\VV(\L)$. The $n$-th root stack of $\L^{\otimes n}$ along $s$ is the quotient stack $[U/\mu_n]$.
\end{lemma}
\begin{proof}
  The result follows from the fact that the following diagram is cartesian:
  \[\xymatrix{
  U\ar[r]\ar[d] & \VV(\L)\ar[d]^{\mu_n\text{-torsor}}\\
  \sqrt[n]{(\L^{\otimes n}, s)/X}\ar[r]\ar[d] & [\VV(\L)/\mu_n]\ar[r]\ar[d] & [\AA^1/\GG_m]\ar[d]^{\hat{}n}\\
  X\ar[r]^{s} & \VV(\L^{\otimes n})\ar[r] & [\AA^1/\GG_m]
  }\]
  To see that $\VV(\L^{\otimes n})\times_{[\AA^1/\GG_m]}[\AA^1/\GG_m]\cong [\VV(\L)/\mu_n]$, we show that the two represent the same fibered category. For a scheme $T$, a $T$-point of the fiber product is given by a tuple $\bigl(f\colon T\to X, \M\in \pic(T), m\in \Gamma(T,\M), \phi\colon \M^{\otimes n}\cong f^*\L^{\otimes n}\bigr)$. Specifying $\M$ and $\phi$ is equivalent to specifying the $n$-torsion bundle $\M\otimes f^*\L^\vee$, which is equivalent to specifying the $\mu_n$-torsor $\pi\colon P\to T$ on which $\pi^*\M$ is naturally isomorphic to $\pi^*f^*\L$. Specifying the section $m$ is then equivalent to specifying the $\mu_n$-invariant section $\pi^*m$ of $\pi^*\M=\pi^*f^*\L$. But a $T$-point of $[\VV(\L)/\mu_n]$ is given precisely by a map $f\colon T\to X$, a $\mu_n$-torsor $\pi\colon P\to T$, and a $\mu_n$-invariant section of $\pi^*f^*\L$.
\end{proof}

\begin{remark}
 Given this description of $n$-th root stacks of $n$-th tensor powers of line bundles, we see that the proof of Theorem \ref{thm:main2} can be interpreted as an application of the philosophy in \cite{kv}. The basic idea of \cite{kv} is as follows. Start with a smooth Deligne-Mumford stack $\X$ and a vector bundle $\E$ on $\X$ so that the stabilizers act faithfully on fibers at geometric points. The goal is to produce a smooth scheme $U$ with a finite surjection onto $\X$. The codimension of the stacky locus in the fibers of the total space of $\E^{\oplus k}$ is large for large $k$, so after repeated slicing by hyperplanes, one can find a slice $U$ of the total space which is smooth, misses the stacky locus, and is finite over $\X$.\footnote{Instead of the total space of $\E^{\oplus k}$, \cite{kv} uses a large fiber product of the projective bundle of $\E$. This is necessary to ensure that enough sections exist. In our situation, we twist by a large power of a very ample line bundle instead.
 }

 In our situation, we aim to produce a smooth scheme $U$ and a finite surjection onto $\X$, which is generically a torsor under some finite group. Our strategy is to start with line bundles $\L_1,\dots, \L_r$ so that the stacky locus of the total space of $\bigoplus \L_i$ has positive codimension in each fiber (this is the condition that the stabilizers act faithfully on the fibers). Then the stacky locus of the total space of $\bigoplus \L_i^{\oplus c_i}$ will have large codimension in each fiber. We then choose slices of this total space which have $\mu_n$-actions so that the intersection of all of the slices is smooth, misses the stacky locus, and is generically a $\mu_n^C$-torsor over $\X$.
 \hfill$\diamond$
\end{remark}

\begin{remark}\label{rmk:slice-stabilizers}
  Suppose $X$ is a stack and $u\in U$ is a geometric point of the $s$-slice of $\VV(\L)$ lying over a point $x\in X$ where $s$ does not vanish. Then $\stab_U(u)$ is given by the kernel of the action of $\stab_X(x)$ on the fiber of $\L$. To see this, first note that $u$ maps to the complement of the zero section in $\VV(\L)$. Since $U\to\VV(\L)$ is a closed immersion, and $\VV(\L)\setminus\{0\}\to\VV(\L)$ is an open immersion, we see $\stab_U(u)=\stab_{\VV(\L)\setminus\{0\}}(u)$. Let $f:X\to B\GG_m$ be the morphism corresponding to $\L$. From the cartesian diagram
  \[\xymatrix{
    \VV(\L)\setminus\{0\}\ar[d]\ar[r] & \ast\ar[d]\\
    X\ar[r]^f & B\GG_m
  }\]
  we see that $\stab_{\VV(\L)\setminus\{0\}}(u) = \stab_X(x)\times_{\GG_m}1$ is the kernel of the action of $\stab_X(x)$ on the fiber $\L(x)$.
\end{remark}

\subsection{Cuts of the coarse space of \texorpdfstring{$\VV\bigl(\bigoplus\O_{X^\can}(D_i)\bigr)$}{V(sum(O(D\_i)))}}\label{subsec:total-spaces}
Returning to the notation from the start of \S\ref{sec:explicit}, in which $X$ is a toric variety with fan $\Sigma$, let $D_1,\dots, D_r$ be (not necessarily distinct) divisors on $X$. For each $i$, suppose $D_i$ is linearly equivalent to $\sum_\rho c_{i\rho}D_\rho$, where $D_\rho$ is the divisor associated to the ray $\rho$ of $\Sigma$. Given a ray $\rho$ of $\Sigma$, let $\lambda_\rho\in N$ be the first lattice point along $\rho$, and define $\hat\rho$ to be the ray in $\ZZ^r\oplus N$ spanned by $(-\sum_i c_{i\rho}e_i, \lambda_\rho)$. Given a cone $\sigma$ of $\Sigma$, define $\hat\sigma$ to be the cone generated by $\{e_1,\dots, e_r\}$ and $\{\hat\rho|\rho$ a ray of $\sigma\}$.  Let $\hhat\Sigma$ be the fan on $\ZZ^r\oplus N$ generated by the cones $\{\hat\sigma|\sigma\in \Sigma\}$. Let $W$ be the toric variety with fan $\hhat\Sigma$, and let $D_i'$ be the divisor in $W$ corresponding to the ray $e_i$. Let $p\colon W\to X$ be the morphism induced by the projection $\hhat\Sigma\to\Sigma$.

This generalizes the usual toric construction of the total space of the sum of line bundles. Informally, $W$ can be thought of as the joint total space of a collection of Weil divisors. This is made precise by the following proposition.

\begin{proposition} \label{prop:weiltotsp}
  Let $\L_i=\O_{X^\can}(D_i)$, as defined in Remark \ref{rmk:canonical-pic=cl}. Then $W$ is the coarse space of $\VV(\bigoplus \L_i)$.
\end{proposition}
\begin{proof}
  Any torus factor in $X$ appears as torus factor in $W$, so we may assume $X$ has no torus factor. Let $\pi:V\to X^\can$ be the $H$-torsor from the Cox construction. By descent, the data of the line bundle $\L_i$ is equivalent to the data of the line bundle $\pi^*\L_i$ with an $H$-linearization. Since $V$ is an open subset of $\AA^{\Sigma(1)}$, $\pi^*\L_i$ is trivial for each $i$. Letting $\hhat V=\VV(\bigoplus \pi^*\L_i)\subseteq \AA^{\Sigma(1)}\times\AA^r=\AA^{\hhat\Sigma(1)}$, we have that $\VV(\bigoplus \L_i)=[\hhat V/H]$, where the action of $H$ on the coordinate corresponding to $\L_i$ is given by $D_i\in \Cl(X)=D(H)$. We will show that this quotient structure naturally agrees with the Cox construction of $W$.

  Every maximal cone of $\hhat\Sigma$ contains all the $e_i$, so the open subset of $\AA^{\hhat\Sigma(1)}$ in the Cox construction of $W$ is determined by sets of rays $\{\hat\rho_1,\dots,\hat\rho_m\}$ which do not lie on a single cone of $\hhat\Sigma$. Such a subset of rays fails to lie on a single cone of $\hhat\Sigma$ if and only if the corresponding subset $\{\rho_1,\dots, \rho_m\}$ fail to lie on a single cone of $\Sigma$. This shows that $\hhat V$ is the same open subset of $\AA^{\hhat\Sigma(1)}$ as in the Cox construction of $W$.
  
  It remains to check that the $H$-action described above agrees with the group action from the Cox construction of $W$. We have a commutative diagram with exact rows
  \[\xymatrix{
    0\ar[r] & N^* \ar[r]^-{\beta^*}\ar[d]^{\iota_1} & (\ZZ^{\Sigma(1)})^*\ar[r]\ar[d]^{\iota_2} & \cok(\beta^*)\ar[r]\ar[d]^{\iota_3} & 0\\
    0\ar[r] & (\ZZ^r)^*\oplus N^* \ar[r]^-{\hat\beta^*} & (\ZZ^r)^*\oplus (\ZZ^{\Sigma(1)})^*\ar[r] & \cok(\hat\beta^*)\ar[r] & 0
  }\]
  where exactness on the left follows from the fact that $X$ has no torus factors (so $\Sigma(1)$ spans $N$). The maps $\iota_1$ and $\iota_2$ are given by inclusions $x\mapsto (0, x)$, and $\hat\beta^*$ induces the identity map on the cokernels, $(\ZZ^r)^*\to (\ZZ^r)^*$, so by the Snake Lemma, $\iota_3$ is an isomorphism. That is, $\cok(\hat\beta^*)\cong \cok(\beta^*)=D(H)$, so the group acting on $\AA^{\hhat\Sigma(1)}$ in the Cox construction of $W$ is $H$. The images of the generators of $(\ZZ^{\Sigma(1)})^*$ (resp.~$(\ZZ^r)^*\oplus(\ZZ^{\Sigma(1)})^*$) in $D(H)$ give the character with which $H$ acts on the corresponding coordinate of $\AA^{\Sigma(1)}$ (resp.~$\AA^r\times \AA^{\Sigma(1)}$). Commutativity of the right square shows that the action of $H$ on the coordinate corresponding to $\hat\rho$ is the same as the action of $H$ on the coordinate corresponding to $\rho$ in $\hhat V$. Finally, since $\hat\beta^*(e_i)=(e_i, -\sum_i c_{i\rho}\lambda_\rho)$ is in the kernel of the map to $D(H)$, we see that the action of $H$ on the coordinate corresponding to $e_i$ is equal to $\sum_i c_{i\rho} D_\rho = D_i\in \Cl(X)=D(H)$.
\end{proof}

Now suppose $n_iD_i$ is a very ample Cartier divisor on $X$, with corresponding lattice polyhedron $P_i = \{x\in N^\vee\otimes \RR| \langle x,\rho\rangle \ge -n_ic_{i\rho}$ for all $\rho\in\Sigma(1)\}$. Consider the polyhedron
\[
  P_i' = \{x\in (\ZZ^r\oplus N)^\vee\otimes \RR| \langle x,\hat\rho\rangle\ge -n_ic_{i\rho}\text{ for all }\rho\in\Sigma(1), \langle x,e_j\rangle=0\text{ for }j\neq i\text{, and }\langle x,e_i\rangle \ge 0\}.
\]
The lattice points in $P_i'$ correspond to torus semi-invariant sections of $\O_W(p^*(n_iD_i))$ which do not vanish along any $D_j'$ with $j\neq i$. We make two observations:
\begin{enumerate}
  \item There is a natural identification of $P_i$ with $\{x\in P_i'|\langle x,e_i\rangle =0\}$. Geometrically, pullback induces an isomorphism between the sections of $\O_X(n_iD_i)$ and the sections of $\O_W(p^*(n_iD_i))$ which are linear combinations of torus semi-invariant sections not vanishing along any of the $D_j'$.
  \item There is a point $x\in P_i'$ with $\langle x,e_i\rangle=n_i$ and $\langle x,\hat\rho\rangle=0$ for all $\rho\in\Sigma(1)$. This follows immediately from the fact that $n_iD_i'$ is linearly equivalent to $p^*(n_iD_i)$, which is true by the construction of $\hhat\Sigma$. Geometrically, this means that there is a section of $\O_W(p^*(n_iD_i))$ which induces the divisor $n_iD_i'$. We will denote this section by $t_i$.
\end{enumerate}
Concretely, $P_i'$ is a pyramid of height $n_i$ with base $P_i$. The apex of the pyramid is the lattice point corresponding to $t_i$. If $s$ is a linear combination of torus semi-invariant sections corresponding to the lattice points in $P_i$, then $p^*(s)$ is ``the same'' linear combination of the torus semi-invariant sections corresponding to the lattice points in the base of $P_i'$.

\begin{definition}\label{def:tv-cut}
  Suppose $s$ is a section of $\O_X(n_iD_i)$. The \emph{$s$-cut of $W$} is the vanishing locus of the section $t_i-p^*(s)$ of $\O_W(p^*(n_iD_i))$.
\end{definition}
\begin{remark}[$\mu_{n_i}$-action on the cut]
  The 1-parameter subgroup corresponding to $e_i$ acts with weight 0 on $p^*(s)$ and with weight $n_i$ on $t_i$, so the $s$-cut of $W$ is invariant under the action of $\mu_{n_i}$.
\end{remark}

\begin{remark}[compatibility of terminology: slices vs.~cuts]\label{rmk:slices-cuts}
  The pyramid of height 1 and base $P_i$ corresponds to (a compactification of) the total space of the line bundle $\O_X(n_iD_i)$. For a section $s$ of $\O_X(n_iD_i)$, the corresponding closed subscheme $X\hookrightarrow \VV(\O_X(n_iD_i))$ is given by the difference of the torus semi-invariant section corresponding to the apex and the pullback of $s$. The $n_i$-th power map $\VV(\O_{X^\can}(D_i))\to \VV(\O_{X^\can}(n_iD_i))$ is induced by vertically scaling the pyramid by a factor of $\frac 1{n_i}$. Combining Proposition \ref{prop:weiltotsp} with the fact that coarse space morphisms are stable under base change, we see that the $s$-cut of $W$ is the coarse space of $U\times \VV(\bigoplus_{j\neq i}\L_i)$, where $U$ is the $s$-slice of $\VV(\L_i)$.
\end{remark}

\subsection{Proof of Theorem \ref{thm:tv}}
\label{subsec:tv}
We now prove the explicit procedure described for toric varieties. We follow the proof of Theorem \ref{thm:main2}, taking into account the results in \S\S\ref{subsec:lb}--\ref{subsec:total-spaces}. Let $X=V/H$ be the Cox construction of $X$, so $\X=[V/H]$ is the canonical stack of $X$. Let $\phi\colon \X\to X$ be the coarse space map, which is an isomorphism away from the singular locus $Z\subseteq X$.

(1) The irreducible torus-invariant divisors of $X$ are $\QQ$-Cartier and clearly generate the class groups of all torus-invariant open affine subvarieties. The $n_iD_i$ can be assumed to be very ample since one can add an arbitrary ample divisor to each $D_i$ without changing the fact that the $D_i$ generate the class groups of all torus-invariant open subvarieties. By Lemma \ref{lem:lb}, the stabilizers of $\X$ act faithfully on the fibers of $\bigoplus \O_\X(D_i)$.

(2) The $s_{i,j}$ may be chosen to have the given properties by the same Bertini argument used in the proof of Theorem \ref{thm:main2}.

(3) Let $\U_{i,j}\subseteq \VV(\O_\X(D_i))$ be the $s_{i,j}$-slice together with the action of $\mu_{n_i}$ (Definition \ref{def:vb-slice}). Let $U\subseteq \VV(\bigoplus_i\O_\X(D_i)^{\oplus c_i})$ be the fiber product of the $\U_{i,j}$ over $\X$, and let $\pi\colon U\to \X$ be the projection. The stacky locus of $\VV(\bigoplus_i\O_\X(D_i)^{\oplus c_i})$ lies over $Z$. Combining Remark \ref{rmk:slice-stabilizers} with the fact that $\bigcap_j V(s_{i,j})$ is disjoint from $Z$, we see that for any geometric point $u$ of $U$, $\stab_U(u)$ must be contained in the kernel of the action of $\stab_\X(\pi(u))$ on $(\bigoplus \O_\X(D_i))_{\pi(u)}$. Since this action is faithful, the stabilizers of geometric points of $U$ are trivial, so $U$ is an algebraic space.

By Lemma \ref{lem:pedroot}, the stack $\U$, given by taking the $n_i$-th root of $s_{i,j}$ for each $i$ and $j$, is isomorphic to $[U/\prod_i\mu_{n_i}^{c_i}]$. By Proposition \ref{prop:weiltotsp}, the coarse space of $\VV(\bigoplus_i\O_\X(D_i)^{\oplus c_i})$ is $W$, and by Remark \ref{rmk:slices-cuts}, the intersection of the $s_{i,j}$-cuts of $W$ (Definition \ref{def:tv-cut}) is $U$.

Since the preimages of $V(s_{i,j})$ in $\X$ are smooth and have simple normal crossings, the root stack $\U$ is smooth, so $U$ is smooth. We therefore have that $U$ is a smooth scheme with an action of $\prod \mu_{n_i}^{c_i}$ such that $X\cong U/\prod \mu_{n_i}^{c_i}$.

\section{Example: blow-up of \texorpdfstring{$\PP(1,1,2)$}{P(1,1,2)}}
\label{sec:ex}
In this section, we run through the procedure described in Theorem \ref{thm:tv} when $X$ is weighted projective space $\PP(1,1,2)$ blown-up at a smooth torus-invariant point.  The fan of $X$ is illustrated below.
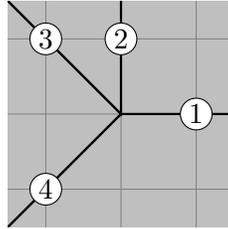
\[\begin{tikzpicture}
 \clip (-1.5,-1.5) rectangle (1.5,1.5);
 \foreach \x/\y/\z/\w in {9/0/0/9, 0/9/-9/9, -9/9/-9/-9, -9/-9/9/0}
   \filldraw [draw=black,fill=lightgray] (\x,\y) -- (0,0) -- (\z,\w);
 \draw[help lines] (-10,-10) grid (10,10);
 \foreach \x/\y/\z/\w in {9/0/0/9, 0/9/-9/9, -9/9/-9/-9, -9/-9/9/0}
   \draw [thick] (\x,\y) -- (0,0) -- (\z,\w);
 \foreach \x/\y/\dottext in {1/0/1, 0/1/2, -1/1/3, -1/-1/4}
   \filldraw[draw=black,fill=white] (\x,\y) node {$\dottext$} circle (6pt);
 \end{tikzpicture}
\]
As mentioned in the introduction, this example is interesting because the map $U\to X$ of Corollary \ref{cor:main} is not toric:
\begin{proposition}
\label{prop:notquot}
There is no smooth toric variety $U$ and finite subgroup $G$ of the torus of $U$ such that $X=U/G$.
\end{proposition}
\begin{proof}
  We show that there is no finite toric morphism $f:U\to X$ with $U$ a smooth toric variety.  Let $\Sigma$ depicted above be the fan of $X$ with ambient lattice $N=\ZZ^2$. Let $\Sigma'$ be the fan of the hypothetical smooth toric variety $U$ with ambient lattice $N' \subseteq N$.  Let $\sigma\in\Sigma'$ be the convex cone with rays $(1,0)$ and $(0,1)$. Let $(a,0)$ and $(0,b)$ be the first lattice points of $N'$ on these rays. Since $\Sigma'$ is assumed to be smooth, these lattice points must form a basis for $N'$ and so $N'=a\ZZ\oplus b\ZZ$. Let $\tau\in\Sigma'$ be the cone generated by $(-1,1)$ and $(-1,-1)$.  We claim that $\tau$ is singular.  Indeed, the first lattice points of $N'$ on the rays of $\tau$ are $(-m, m)$ and $(-m,-m)$, where $m=ab/\gcd(a,b)$ is the least common multiple of $a$ and $b$. These do not form a basis for $N'$ since $(a,0)$ is not in their $\ZZ$-span.
\end{proof}

We now apply Theorem \ref{thm:tv} to $X$.

\textbf{(1: choose $D$)} Let $D_{\rho_i}$ be the divisor corresponding to the $i$-th ray. We seek an ample divisor $D=\sum a_i D_{\rho_i}$ which generates the class group of each torus-invariant open affine. The only non-trivial class group is $\ZZ/2$ (for the open neighborhood of the singular point), so we have the condition that $a_3+a_4$ must be odd. Recall that the polytope of $D$ is $\{x\in \RR^2|\langle x,\rho_i\rangle \ge -a_i\}$, that $D$ is ample if and only if the dual fan to its polytope is $\Sigma$, and that $D$ is Cartier if and only if the vertices of its polytope are lattice points. From these facts, we see that $D=D_{\rho_1}+D_{\rho_2}+D_{\rho_3}$ generates the class group of each torus-invariant open affine, and that $2D$ is Cartier and very ample (see its polytope pictured below).
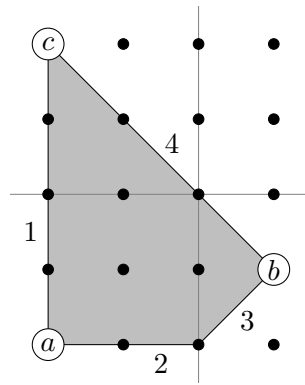
\[\begin{tikzpicture}
  \clip (-2.5,-2.5) rectangle (1.5,2.5);
  \filldraw[fill=lightgray] (-2,-2) -- (-2,2) -- (1,-1) -- (0,-2) -- cycle;
  \draw[help lines] (-3,0) -- (3,0);
  \draw[help lines] (0,-3) -- (0,3);
  \foreach \x in {-2,-1,0,1,2}
    \foreach \y in {-2,-1,0,1,2}
       \draw plot[mark=*] coordinates{(\x,\y)};
 \foreach \x/\y/\dottext in {-2/-2/a, 1/-1/b, -2/2/c}
   \filldraw[draw=black,fill=white] (\x,\y) node {$\dottext$} circle (6pt);
 \draw (-2,-0.5) node[left] {1};
 \draw (-0.5,-2) node[below] {2};
 \draw (0.5,-1.5) node[below right=-0.5ex] {3};
 \draw (-0.5,0.5) node[above right=-0.5ex] {4};
\end{tikzpicture}
\]
The edges of this polytope correspond to the torus-invariant divisors, and the lattice points correspond to semi-invariant sections of $2D$. The section corresponding to a lattice point $v$ vanishes along $D_i$ to order $\langle v, \rho_i\rangle + a_i$ (the ``distance of the lattice point from the corresponding edge''), so when it is pulled back to the the open subset $V\subseteq \AA^4 = \spec k[x_1,x_2,x_3,x_4]$ in the Cox construction, it becomes $\prod_i x_i^{\langle v, \rho_i\rangle + a_i}$.

\textbf{(2: choose $s$)} We next choose a section $s$ of $\O(2D)$ whose vanishing locus is smooth and misses the singular point of $X$, the intersection of $D_{\rho_3}$ and $D_{\rho_4}$. Let $s_a$, $s_b$, and $s_c$ be the torus semi-invariant sections of $\O(2D)$ corresponding to the labeled lattice points. Note that $s_a$, $s_b$, and $s_c$ pull back to $x_3^2 x_4^4$, $x_1^3 x_2$, and $x_2^4 x_3^6$ in the Cox cover, respectively. Since $s_b$ is the only semi-invariant section not vanishing at the singular point, it must appear in $s$, but $s_b$ vanishes on both $D_{\rho_1}$ and $D_{\rho_2}$, so its vanishing locus is singular. We try setting $s=s_a+s_b$ to remedy this, but by the Jacobian criterion, $x_3^2x_4^4 + x_1^3x_2$ has critical points when $x_1$ and $x_4$ both vanish. So we try $s=s_a+s_b+s_c$; by the Jacobian criterion, $x_3^2x_4^4 + x_1^3x_2 + x_2^4x_3^6$ has smooth vanishing locus, so $s$ has smooth vanishing locus.


\textbf{(3: compute $U$)} Lastly, we compute an explicit smooth $U$ and finite group $G$ such that $X=U/G$. Consider the projective toric variety in $\PP^{18}$ defined by the following polytope. 
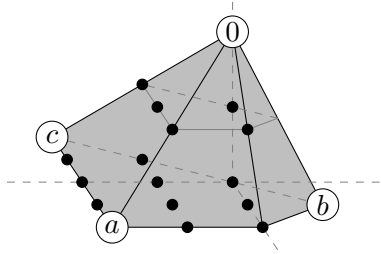
\[\begin{tikzpicture}
  \clip (-3,-1.5) rectangle (2,2.5);
  \pgfsetzvec{\pgfpointxy{0}{1}}
  \pgfsetyvec{\pgfpointxy{-.2}{.3}}
  \filldraw[fill=lightgray] (0,0,2) -- (-2,2,0) -- (-2,-2,0) -- (0,-2,0) -- (1,-1,0) -- cycle;
  \draw (-2,-2,0) -- (0,0,2) -- (0,-2,0);
  \draw[help lines,dashed] (-2,2,0) -- (1,-1,0);
  \draw[help lines] (-1,1,1) -- (-1,-1,1) -- (0,-1,1) -- (.5,-.5,1);
  \draw[help lines,dashed] (.5,-.5,1) -- (-1,1,1);
  \draw[help lines,dashed] (0,-3,0) -- (0,0,0) -- (0,0,3);
  \draw[help lines,dashed] (-3,0,0) -- (3,0,0);
  \draw plot[only marks,mark=*] coordinates{
        (-2,-2,0) (-2,-1,0) (-2,0,0) (-2,1,0) (-2,2,0)
        (-1,-2,0) (-1,-1,0) (-1,0,0) (-1,1,0)
        (0,-2,0)  (0,-1,0) (0,0,0)
                  (1,-1,0)
        (-1,-1,1) (-1,0,1) (-1,1,1)
        (0,-1,1) (0,0,1) (0,0,2)};
 \foreach \x/\y/\dottext in {-2/-2/a, 1/-1/b, -2/2/c}
   \filldraw[draw=black,fill=white] (\x,\y,0) node {$\dottext$} circle (6pt);
 \filldraw[draw=black,fill=white] (0,0,2) node {$0$} circle (6pt);
\end{tikzpicture}\]
Note that the toric variety $W$ defined by $\hhat\Sigma$ is an open subvariety, and this projective variety is precisely the closed image of the morphism $\psi\colon W\to \PP^{18}$ given by $2D'$. Let $U$ be the hyperplane slice of this variety (or equivalently, of $W$) defined by $x_0-x_a-x_b-x_c$.

There is a $\mu_2$-action on $\PP^{18}$ given as follows: if $\chi$ is the non-trivial character of $\mu_2$ and $x$ is a lattice point of the above polytope with height $h$, then $\mu_2$ acts the coordinate corresponding to $x$ through the character $\chi^h$. Then $U$ is smooth, is invariant under the $\mu_2$-action, and $X=U/\mu_2$.

\section{Characterization of torus quotients: proof of main theorem}
\label{sec:criterion}
Throughout this section, we work over a field $k$.  Our goal is to prove Theorem \ref{thm:criterion}. In \S\ref{subsec:qab}, we use Theorem \ref{thm:criterion} to prove Theorem \ref{thm:P2A5}, thereby answering Question \ref{q:abelian}.

\begin{notation}
  Throughout this section, we will use the notation $\hhat\O_{X,x}$ to denote the completion of the \emph{\'etale} local ring at a point $x$ of an algebraic space $X$.\footnote{The \'etale local ring at $x$ is simply the strict henselization of the local ring at a preimage of $x$ in some \'etale cover of $X$. See \cite[Definition \href{http://math.columbia.edu/algebraic_geometry/stacks-git/locate.php?tag=04KG}{04KG}]{stacks-project}.} Note that if $X$ is a scheme and the residue field at $x$ is algebraically closed, this agrees with the completion of the Zariski local ring.\footnote{If $X$ is a scheme and the residue fields are not algebraically closed, the proofs in this section apply verbatim if one uses the completion of the Zariski local ring. The \'etale local ring is used simply because algebraic spaces do not have Zariski local rings.}

  We use the term \emph{torus} to mean split torus. We say that a stack $\X$ is a \emph{torus quotient stack} if there is a smooth algebraic space $V$ and a torus $T$ acting properly on $V$ with finite stabilizers such that $\X=[V/T]$. We say an algebraic space $X$ is a \emph{torus quotient space} if it is the coarse space of a torus quotient stack.
\end{notation}

\begin{theorem}
\label{thm:criterion}
  Suppose an algebraic space $X$ has tame diagonalizable quotient singularities. Consider the following conditions:
  \begin{enumerate}
    \item \label{crit:fin-quot} $X$ is a quotient of a smooth algebraic space by a proper action of a finite diagonalizable group.
    \item \label{crit:tor-quot} $X$ is a torus quotient space.
    \item \label{crit:div} $X$ has Weil divisors $D_1,\dots, D_r$ whose images generate $\Cl(\hhat\O_{X,x})$ for all points $x$ of $X$.
    \item \label{crit:can} $X^\can$ is a torus quotient stack.
  \end{enumerate}
  Conditions $(\ref{crit:tor-quot})$--\,$(\ref{crit:can})$ are equivalent and implied by $(\ref{crit:fin-quot})$.  If $X$ is quasi-projective and $k$ is infinite, then $(\ref{crit:fin-quot})$ is equivalent to $(\ref{crit:tor-quot})$--\,$(\ref{crit:can})$.
\end{theorem}
\begin{remark}\label{rmk:criterion-zariski-local}
  Condition (\ref{crit:div}) is a Zariski local, as the image of $\Cl(X)$ in $\Cl(\hhat\O_{X,x})$ is determined by any open neighborhood of $x$, so the other conditions are also Zariski local.
\end{remark}

\begin{proof}[\ifnotlms Proof \fi of Theorem \ref{thm:main-intro} from Theorem \ref{thm:criterion}]
Note that if $X$ is a quasi-projective variety and $k$ is algebraically closed, then tame finite diagonalizable groups are the same as tame finite abelian groups. It therefore suffices to show that condition ($i$) of Theorem \ref{thm:main-intro} is equivalent to condition ($i$) of Theorem \ref{thm:criterion} for $1\leq i\leq 4$. Condition (\ref{crit:div}) is the same in the two theorems.

  We first show that the conditions of Theorem \ref{thm:main-intro} imply those of Theorem \ref{thm:criterion}. An action of a finite group on a separated scheme is always proper by Lemma \ref{lem:quotient=coarse-space}(\ref{item:finite->proper}), so condition (\ref{introitem:finite-quot}) of Theorem \ref{thm:main-intro} implies condition (\ref{crit:fin-quot}) of Theorem \ref{thm:criterion}. If condition (\ref{introitem:torus-quot}) of Theorem \ref{thm:main-intro} holds, then $X$ is the geometric quotient of a quasi-projective variety $V$ by a torus $T$.
  By Lemma \ref{lem:quotient=coarse-space}(\ref{item:q-proj->quotients-agree}), the action of $T$ on $V$ is proper and $X$ is the coarse space of $[V/T]$, so condition (\ref{crit:tor-quot}) of Theorem \ref{thm:criterion} holds. Lastly, since canonical stack morphisms are proper, the canonical stack over a separated scheme is separated. If the canonical stack over $X$ is a quotient of a quasi-projective $V$ by a torus $T$, then the action must therefore be proper, so condition (\ref{introitem:canonical}) of Theorem \ref{thm:main-intro} implies condition (\ref{crit:can}) of Theorem \ref{thm:criterion}. 
  
  We next show that the conditions of Theorem \ref{thm:criterion} imply those of Theorem \ref{thm:main-intro}. If $V$ is a smooth algebraic space with a proper action of an affine group $H$ with tame finite stabilizers so that $X=V/H$, then by Lemma \ref{lem:quotient=coarse-space}(\ref{item:algsp->q-proj-var}), $V$ is a quasi-projective variety. Therefore, for $i=\ref{crit:fin-quot},\ref{crit:can}$, condition ($i$) of Theorem \ref{thm:criterion} implies condition ($i$) of Theorem \ref{thm:main-intro}. By Lemma \ref{lem:quotient=coarse-space} (\ref{item:algsp->q-proj-var} then \ref{item:q-proj->quotients-agree}), condition (\ref{crit:tor-quot}) of Theorem \ref{thm:criterion} implies condition (\ref{introitem:torus-quot}) of Theorem \ref{thm:main-intro}.
\end{proof}

We now assemble the results needed to prove Theorem \ref{thm:criterion}.
\begin{lemma}
\label{l:diagonalizable-complete-local}
  Let $G$ be a finite diagonalizable group acting on a regular complete noetherian local $k$-algebra $A$ with maximal ideal $\m$ and residue field $K$. There is a $G$-equivariant isomorphism $\bigl(\sym^*_K(\m/\m^2)\bigr)^\wedge\to A$, where $\m/\m^2$ is given the induced linear $G$-action and $(-)^\wedge$ denotes completion. Moreover, $A^G$ is isomorphic to the complete local ring at the distinguished point of a pointed affine toric variety over $K$.
\end{lemma}
\begin{proof}
  Since $G$ is linearly reductive (even if it is not tame), the surjection $\m\to \m/\m^2$ admits a $G$-equivariant splitting $\m/\m^2\to \m$. The induced ring homomorphism $\bigl(\sym^*_K(\m/\m^2)\bigr)^\wedge\to A$ is $G$-equivariant by construction, and induces an isomorphism of tangent spaces, so by the Cohen Structure Theorem \cite[Theorem 7.7]{eisenbud}, it is an isomorphism.

  Since $G$ is linearly reductive, taking $G$-invariants commutes with limits, so the completion of $\sym^*_K(\m/\m^2)^G$ is isomorphic to $\bigl((\sym^*_K(\m/\m^2))^\wedge\bigr)^G \cong A^G$. Thus, for the final assertion, it suffices to show that $\sym^*_K(\m/\m^2)^G$ is the coordinate ring of a pointed affine toric variety. Since $G$ is diagonalizable, there exists a basis $x_1,\dots, x_n$ for $\m/\m^2$ so that each $x_i$ is semi-invariant. Let $M$ be the monoid of monomials in the $x_i$ which are invariant under $G$. It is clear that $\sym^*_K(\m/\m^2)^G = K[x_1,\dots, x_n]^G=K[M]$, and that $\spec K[M]$ is a pointed toric variety.
\end{proof}

\begin{lemma}
\label{l:diagsnc}
  Let $U$ be a smooth algebraic space and $G$ a diagonalizable group scheme which acts properly on $U$ with tame finite stabilizers. Let $\X$ be the canonical stack over $U/G$ and $f\colon [U/G]\to\X$ the induced map (see Remark \ref{rmk:canonical-universal-property}). Then the ramification divisor $\D\subseteq \X$ of $f$ is a simple normal crossing divisor with smooth components.
\end{lemma}
\begin{proof}
  Let $X=U/G$ be the coarse space of $[U/G]$. To show that $\D$ has simple normal crossings, it suffices to check that for every point $x$ of $X$, the pullback of $\D$ to $\X\times_X\spec\hhat\O_{X,x}$ is a simple normal crossing divisor. We may therefore replace $X$ by $\spec\hhat\O_{X,x}$, $U$ by $\spec\hhat\O_{U,u}$, and $G$ by the stabilizer of $u$, where $u$ maps to $x$. We may quotient by the kernel of the action of $G$ without changing the ramification divisor $\D$, so we may assume $G$ acts faithfully on $U$.

  Let $\m$ be the maximal ideal of $\O_{U,u}$ and let $K$ be the residue field of $u$. By Lemma \ref{l:diagonalizable-complete-local}, $\O_{U,u}$ is $G$-equivariantly isomorphic to $\bigl(\sym^*_K(\m/\m^2)\bigr)^\wedge$. Since $G$ is diagonalizable, there is a basis $x_1,\dots, x_n$ for $\m/\m^2$ so that each $x_i$ is semi-invariant. Let $P_i\subseteq G$ be the subgroup which acts trivially on $x_j$ for all $j\neq i$. The $P_i$ are precisely the pseudoreflections of the action of $G$. The subgroup $P\subseteq G$ generated by pseudoreflections is the direct sum of the $P_i$. By Remark \ref{rmk:canonical-local-rings}, we have
  \[
    \X=[\spec\O_{U,u}^P/(G/P)].
  \]
  Since $P$ and $G/P$ are \'etale, and $\D$ is the ramification divisor of
  \[
    [U/G] = [\spec\O_{U,u}/G]=\bigl[[\spec\O_{U,u}/P]\bigm/(G/P)\bigr]\to \bigl[\spec\O_{U,u}^P\bigm/(G/P)\bigr]=\X,
  \]
  it suffices to check that the ramification divisor $D\subseteq \spec\O_{U,u}^P$ of $\pi:\spec\O_{U,u}\to\spec\O_{U,u}^P$ has simple normal crossings. If $P_i$ has order $\ell_i$, then we see that $\pi$ is given by the inclusion of rings $K[[y_1,\dots,y_n]]\subseteq K[[x_1,\dots,x_n]]$ with $y_i=x_i^{\ell_i}$.  Hence, $D$ is the union of the $y_i$-coordinate hyperplanes with $\ell_i\neq 1$, so is a normal crossing divisor.

  Before proving smoothness of the components of $\D$, we observe that the ramification divisor on the source of $\pi$ is $\pi^{-1}(D)\subseteq\spec\O_{U,u}$, given by the $x_i$-coordinate hyperplanes with $\ell_i\neq 1$. Note that the stabilizers of the components are precisely the $P_i$, and that these subgroups of $G$ are distinct for distinct components of the ramification divisor.

  We now show that the components of $\D$ are smooth. For this, it no longer suffices to replace $X$ by $\spec\hhat\O_{X,x}$, so we return to the original notation. To complete the proof, we show that for every point $z$ of $\X$, the number of components of $\D$ passing through $z$ is equal to the number of components of the induced divisor in $\X\times_X\spec\hhat\O_{X,x}$, where $x$ is the image of $z$ in $X$. This shows that no divisor self-intersects.

  \[\xymatrix{
    [U/G]\times_X \spec\hhat\O_{X,x} \ar[r]\ar[d] & [U/G]\ar[d]^f\\
    \X\times_X \spec\hhat\O_{X,x} \ar[r] & \X
  }\qquad\xymatrix{
    \{\text{comp.~of induced divisor}\}\ar[r]^-\sim\ar[d]^\wr & \{\text{comp.~of }f^{-1}(\D)\text{ at }u\}\ar[d]^\wr\\
    \{\text{comp.~of induced divisor}\}\ar[r] & \{\text{comp.~of }\D\text{ at }z\}
  }\]

  Let $u$ be a point of $[U/G]$ such that $f(u)=z$. We have that $f$ is a homeomorphism, as $[U/G]\to X$ and $\X\to X$ are coarse space maps, hence homeomorphisms. Therefore, the components of $\D$ passing through $z$ are in bijection with the components of $f^{-1}(\D)$ passing through $u$. These are in bijection with the components of the induced divisor in $[U/G]\times_X\spec\hhat\O_{X,x}$, as the distinct formal components are stabilized by distinct subgroups of $G$. Finally, these are in bijection with the components of the induced divisor in $\X\times_X\spec\hhat\O_{X,x}$, as the morphism $[U/G]\times_X\spec\hhat\O_{X,x}\to \X\times_X\spec\hhat\O_{X,x}$ is a homeomorphism.
\end{proof}

\begin{corollary}\label{cor:diagrootstack}
  Suppose $U$ is a smooth algebraic space and $G$ is a diagonalizable group that acts properly on $U$ with tame finite stabilizers.  Then the induced map $f:[U/G]\to (U/G)^\can$ is a root stack morphism along a collection of smooth connected divisors with simple normal crossings.
\end{corollary}
\begin{proof}
  By \cite[Theorem 1]{bottomupDM}, any smooth separated tame Deligne-Mumford stack $\X$ is isomorphic to $\sqrt{\D/X^\can}^\can$, where $X$ is the coarse space of $\X$ and $\D$ is the ramification divisor of the induced map $\X\to X^\can$. The result follows from Lemma \ref{l:diagsnc}, as a root stack of a simple normal crossing divisor is smooth \cite[1.3.b(3)]{fmn}, so isomorphic to its canonical stack.
\end{proof}

\begin{corollary}
\label{cor:torusquot}
  Let $X$ be an algebraic space with quotient singularities. Then $X$ is a torus quotient space if and only if $X^\can$ is a torus quotient stack.
\end{corollary}
\begin{proof}
  By Corollary \ref{cor:diagrootstack}, it suffices to show that a stack is a torus quotient stack if and only if its roots along smooth connected divisors are torus quotient stacks. Let $\X$ be a stack, and let $\Y=\sqrt[n]{(\L,s)/\X}$, for some $\L$, $s$, and $n$, where the vanishing locus of $s$ is smooth and connected. Consider the following cartesian diagram, in which $f$ corresponds to the line bundle $\L$ with section $s$. Let $\M$ be the universal line bundle on $\Y$ so that $\M^{\otimes n}\cong \L$ (i.e.~the line bundle corresponding to $u\circ g$ in the diagram below), and let $t$ be the universal section of $\M$ so that $t^{\otimes n}=s$.
  \[\xymatrix{
   \Y\ar[r]^-g\ar[d]_\pi & [\AA^1/\GG_m]\ar[r]^-u\ar[d]^{\hat{}n} & B\GG_m\\
   \X\ar[r]^-f & [\AA^1/\GG_m]\ar[r]^-u & B\GG_m
  }\]

  $(\Rightarrow)$ Suppose $\X$ is a torus quotient stack. By Lemma \ref{lem:quot<->representable} there exists a representable morphism $\kappa:\X\to B\GG_m^r$. By Corollary \ref{cor:representability}(\ref{rep:product}), $(\kappa, f)\colon \X\to B\GG_m^r\times [\AA^1/\GG_m]$ is also representable. Since $(\kappa\circ\pi, g)\colon \Y\to B\GG_m^r\times [\AA^1/\GG_m]$ is the pullback of $(\kappa, f)$, it is representable by Corollary \ref{cor:representability}(\ref{rep:local}). Then $(\kappa, u\circ g)\colon \Y\to B\GG_m^{r+1}$ is representable by Corollary \ref{cor:representability}(\ref{rep:property-P}), so $\Y$ is a torus quotient stack by Lemma \ref{lem:quot<->representable}.

  $(\Leftarrow)$ Now suppose $\Y$ is a torus quotient stack. Then there is a representable morphism $\ell\colon \Y\to B\GG_m^r$, corresponding to a tuple of line bundles $(\L_1,\dots, \L_r)$. Then $(\ell, g)$, corresponding to $(\L_1,\dots, \L_r,(\M,t))$, is representable by Corollary \ref{cor:representability}(\ref{rep:product}). By \cite[Corollary 3.1.2]{root}, since $s$ has connected vanishing locus, we have
  \[
  \L_j=\M^{i_j}\otimes \pi^*\K_j
  \]
  for some $0\le i_j<n$ and line bundle $\K_j$ on $\X$. Composing with the automorphism of $B\GG_m^r\times [\AA^1/\GG_m]$ sending $(\L_1,\dots, \L_r,(\M,t))$ to $(\L_1\otimes\M^{-i_1},\dots \L_r\otimes \M^{-i_r},(\M,t))$, we see that there is a representable morphism of the form $(\kappa\circ \pi, g)\colon \Y\to B\GG_m^r\times [\AA^1/\GG_m]$. This morphism is the pullback of $(\kappa, f)\colon \X\to B\GG_m^r\times [\AA^1/\GG_m]$, which must also be representable by Corollary \ref{cor:representability}(\ref{rep:local}). Then by Corollary \ref{cor:representability}(\ref{rep:property-P}), $(\kappa, u\circ f)\colon \X\to B\GG_m^{r+1}$ is representable, so $\X$ is a torus quotient stack.
\end{proof}

\begin{proposition}\label{prop:inv-class-gp}
  Let $A$ be a regular ring and let $G$ be a tame finite diagonalizable group which acts faithfully on $X=\spec A$ and freely on an open subscheme $U\subseteq X$  whose complement has codimension at least 2. Suppose that one of the following holds:
  \begin{enumerate}
    \item $A$ is local with maximal ideal $\m$, or
    \item $A=k[x_1,\dots,x_n]$ and $G$ fixes the origin in $\AA^n=\spec A$.
  \end{enumerate}
  Then $\Cl(A^G)$ is canonically isomorphic to the Cartier dual $D(G)$ of $G$.
\end{proposition}
\begin{proof}
  By assumption, the complement of $U$ in $X$ is of codimension at least 2, so the canonical stack over $\spec A^G$ is $[X/G]$ by Remark \ref{rmk:canonical-universal-property}. By Remark \ref{rmk:canonical-pic=cl}, we have $\Cl(A^G)\cong \pic([X/G])\cong \pic^G(X)$. We will show that $\pic^G(X)$ is canonically isomorphic to $D(G)$.

  We have that $\pic(X)=0$, so we wish to find all $G$-linearizations of $\O_X$. Twisting the trivial linearization by an element of $D(G)$ gives us a morphism $\iota\colon D(G)\to \pic^G(X)$. Given a linearization, we recover a character of $G$ by considering the action on the fiber over the closed point of $X$ (resp.~the origin), so $\iota$ is injective.

  Now we show $\iota$ is surjective. Let $M$ be a free $A$-module of rank 1. A $G$-linearization of $M$ is equivalent to a $D(G)$-grading $M=\bigoplus M_\chi$ compatible with the $D(G)$-grading on $A$ given by the $G$-action. To prove that this $G$-linearization is in the image of $\iota$, it suffices to find a semi-invariant generator of $M$.

  Let $m\in M$ be a generator, and let $m=\sum m_\chi$ with $m_\chi\in M_\chi$. In the case where $A$ is local with maximal ideal $\m$, not every $m_\chi$ can be in $\m\cdot m$, so some $m_\chi$ is a unit multiple of $m$. This $m_\chi$ is then a semi-invariant generator of $M$.

  In the case where $A$ is a polynomial ring, we claim that $m = m_\chi$ for some $\chi$. First suppose that $k$ has enough roots of unity that $G$ is a discrete group. For any $g\in G$, $g\cdot m$ is a generator, so it must be of the form $u_gm$ for $u_g\in A^\times = k^\times$, showing $m$ is semi-invariant (i.e.~is contained the $M_\chi$ for which $\chi(g)=u_g$). If $G$ is not discrete, let $K$ be an \'etale extension of $k$ so that the pullback $G_K$ is discrete. In $M\otimes_k K$, we then have that $m\otimes 1 = m_\chi\otimes 1$ for some $\chi$. Since $K$ is faithfully flat over $k$, $m = m_\chi$.
\end{proof}
\begin{corollary}\label{cor:pointed-affine-restriction}
  Suppose $X$ a pointed affine toric variety with quotient singularities and distinguished point $x$. Then the restriction morphism $\Cl(X)\to \Cl(\hhat\O_{X,x})$ is an isomorphism.
\end{corollary}
\begin{proof}
  Let $X=V/H$ be the Cox construction of $X$. By Proposition \ref{prop:inv-class-gp}, $\Cl(X)$ and $\Cl(\hhat\O_{X,x})$ are naturally isomorphic to $D(H)$, where a divisor $D$ corresponds to the character of $H$ given by its action on the central fiber of $\O_{X^\can}(D)$. Therefore, we see that the map $\Cl(X)\to \Cl(\hhat\O_{X,x})$ is an isomorphism.
\end{proof}

\begin{proof}[\ifnotlms Proof \fi of Theorem \ref{thm:criterion}]
  To see that (\ref{crit:fin-quot}) implies (\ref{crit:tor-quot}), suppose $X=U/G$ with $U$ smooth and $G\subseteq \GG_m^r$ a finite diagonalizable group scheme. Let $V= (U\times \GG_m^r)/G$, where $G$ acts by $g\cdot (u,t) = (g\cdot u, g^{-1}\cdot t)$. Since $U\times \GG_m^r$ is smooth and the $G$-action is free, $V$ is smooth, and we have $X=V/\GG_m^r$. Conversely, if $X$ is quasi-projective and $k$ is infinite, then Theorem \ref{thm:main2} shows that (\ref{crit:tor-quot}) implies (\ref{crit:fin-quot}).

  Corollary \ref{cor:torusquot} shows that (\ref{crit:tor-quot}) and (\ref{crit:can}) are equivalent.

  To complete the proof, we show that (\ref{crit:div}) and (\ref{crit:can}) are equivalent. By Remark \ref{rmk:canonical-pic=cl} there is an equivalence between Weil divisors on $X$ and line bundles on $X^\can$. By Proposition \ref{prop:quot<->bundle}, $X^\can$ is a torus quotient stack if and only if there is a collection of Weil divisors $D_1,\dots,D_r$ on $X$ such that the residual representations of $\bigoplus\O_{X^\can}(D_i)$ are faithful.  To check faithfulness at a geometric point $z$ of $X^\can$, it suffices to do so after base changing by $\spec \hhat\O_{X,x}\to X$, where $x\in X$ is the image of $z$.

  By Remark \ref{rmk:canonical-local-rings}, we have
  \[
   X^\can\times_X \spec\hhat\O_{X,x}=[\spec R / G]
  \]
  where $R$ is a complete local ring and $G$ acts freely away from a codimension 2 closed subscheme of $\spec R$. Let $\m$ be the maximal ideal of $R$. By Lemma \ref{l:diagonalizable-complete-local}, $\hhat\O_{X,x}$ is isomorphic the complete local ring of a pointed affine toric variety $Y$ at its distinguished point $y$. By Remark \ref{rmk:canonical-local-rings}, we have a cartesian diagram
  \[
  \xymatrix{
  X^\can\ar[d] & [\spec R/ G]\ar[l]\ar[r]\ar[d] & Y^\can\ar[d]\\
  X\ & \spec\hhat\O_{X,x}\ar[l]\ar[r] & Y
  }
  \]
  By Corollary \ref{cor:pointed-affine-restriction}, $\Cl(\hhat\O_{X,x})\cong \Cl(Y)$, so there is a divisor $D'_i$ on $Y$ whose class has the same image in $\Cl(\hhat \O_{X,x})$ as that of $D_i$, and the images of the $D_i$ generate $\Cl(\hhat \O_{X,x})$ if and only if the $D_i'$ generate $\Cl(Y)$. The residual representation of $\bigoplus\O_{X^\can}(D_i)$ at $z$ is faithful if and only if the residual representation of $\bigoplus\O_{Y^\can}(D'_i)$ at $z$ is faithful. Since $y$ is the torus-invariant point of $Y$, this is equivalent to faithfulness of all residual representations $\bigoplus\O_{Y^\can}(D'_i)$, which by Lemma \ref{lem:lb}, holds exactly when the $D_i'$ generate $\Cl(Y)$.
\end{proof}

\begin{example}[The revolved cuspidal cubic is a $\mu_2$-quotient]
  Consider the surface $X\subseteq \AA^3_\CC$ cut out by $x^2+y^2 = z^2(z-1)$, the surface of revolution of the cuspidal cubic:
  
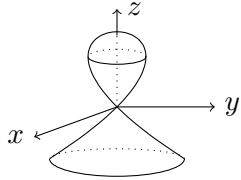
\[\begin{tikzpicture}
    \draw[smooth,domain=-1.3:1.3,variable=\t] plot (\t*\t*\t-\t, 1- \t*\t);
    \draw[->] (0,0) -- (1.3,0) node[right] {$y$};
    \draw[dotted] (0,0) -- (0, 0.95);
    \draw[->] (0, 0.95) -- (0,1.3) node[right] {$z$};
    \draw[->] (0, 0) -- (-1.1,-.4) node[left] {$x$};
    \draw[dotted] (.385,.667) arc (0:180:.385 and .1);
    \draw (-.385,.667) arc (-180:0:.385 and .1);
    \draw[dotted] (.897,-.69) arc (0:180:.897 and .133);
    \draw (-.897,-.69) arc (-180:0:.897 and .233);
  \end{tikzpicture}\]
  The completed local ring at the singular point agrees with the completed local ring at the singular point of the toric variety $x^2+y^2=-z^2$, so by Corollary \ref{cor:pointed-affine-restriction}, we have that the formal local class group is $\ZZ/2$ at the cone point, and trivial elsewhere. The Weil divisor $D$ cut out by the ideal $(x-iy,z)$ is clearly non-principal at the origin, so it generates the formal local class group. Since $2D$ is Cartier, the proof of Theorem \ref{thm:criterion} shows that $X$ is expressible as a $\mu_2$-quotient of a smooth variety.
%
%
%
\end{example}

\begin{example}[Strict toroidal embeddings are finite diagonalizable group quotients]
  Recall that a \emph{toroidal embedding} consists of a scheme $X$ and an open subset $W\subseteq X$ such that for every closed point $x\in X$, there exists a toric variety $Y$ with torus $T$, a point $y\in Y$ and an isomorphism of complete local rings $\hhat\O_{X,x}\to\hhat\O_{Y,y}$ sending the ideal of $X\setminus W$ to the ideal of $Y\setminus T$. The toroidal embedding $W\subseteq X$ is said to be \emph{strict} if every irreducible component of $X\setminus W$ is a normal divisor.

  Suppose $W\subseteq X$ is a strict toroidal embedding, with $X$ a quasi-projective variety with (automatically diagonalizable) quotient singularities over an infinite field. Let $D_1,\dots,D_r$ be the irreducible components of $X\setminus W$. Given any point $x\in X$, choose a toric variety $Y$, a point $y\in Y$, and an isomorphism $\hhat\O_{X,x}\to\hhat\O_{Y,y}$ as above. Since $W\subseteq X$ is strict, the images of the $D_i$ are linearly equivalent to the images of the torus-invariant divisors of $Y$, and therefore generate the formal local class groups by Corollary \ref{cor:pointed-affine-restriction}. Theorem \ref{thm:criterion} then shows that $X$ is a quotient of a smooth scheme by a finite diagonalizable group.

  This was an implicit result of \cite[\S 5.3]{abramovich-karu}.
\end{example}

\subsection{Proof of Theorem 1.5: answering Question \ref{q:abelian} negatively}
\label{subsec:qab}

In this subsection, we prove Theorem \ref{thm:P2A5}. Suppose $X=\PP(V)/A_5$ is of the form $U/G$ for $U$ a smooth variety and $G$ a finite abelian group. Let $p=\mathrm{char}(k)$. At any point of $U$, the $p$-part $P\subseteq G$ must act by pseudoreflections, as the orders of the stabilizers of the canonical stack $[\PP(V)/A_5]$ have no $p$-part (see Remark \ref{rmk:canonical-universal-property}). It follows that $U/P$ is smooth. Since $X=(U/P)/(G/P)$, we may assume $G$ has no $p$-part. That is, to prove part (1) of Theorem \ref{thm:P2A5}, it suffices to show $X$ is not the quotient of a smooth variety by a tame finite diagonalizable group.

The action of $A_5$ on $V$ induces a linearization of $\O_{\PP(V)}(1)$. Let $\L$ denote this linearized line bundle. Letting $\alpha\colon A_5\times \PP(V)\to \PP(V)$ be the action and $p\colon A_5\times \PP(V)\to \PP(V)$ be the projection, we let $\phi\colon \alpha^*\O(1)\cong p^*\O(1)$ be the descent data corresponding to the linearization $\L$.

We show that $\L^{\otimes n}$ is the unique $A_5$-linearization of $\O(n)$. Suppose $\psi\colon \alpha^*\O(n)\cong p^*\O(n)$ is a choice of descent data corresponding to a linearization of $\O(n)$. Then $\psi/\phi^n:A_5\times \PP(V)\to \GG_m$ is given by a regular function $\chi\colon A_5\to \GG_m$, as regular functions on the components of $A_5\times \PP(V)$ are constant. The cocycle condition $f(gh,x)=f(h,x)\cdot f(g,h\cdot x)$ is satisfied by both $\phi^n$ and $\psi$, so it is satisfied by $\psi/\phi^n$. Thus, $\chi$ is a character of $A_5$, so it is trivial, so $\psi=\phi^n$ as desired.

To show that $X$ is not a quotient of a smooth variety by a tame finite diagonalizable group, by Theorem \ref{thm:criterion}, it is enough to show that every line bundle on the canonical stack $\X=[\PP(V)/A_5]$ has trivial residual representations.  If a line through the origin of $V$ is fixed by an element $g\in A_5$, then $g$ acts by rotation about that line. As a result, the stabilizers of points in $\PP(V)$ act trivially on the fibers of $\L$.  Since $\O(n)$ has a unique choice of linearization given by $\L^n$, the residual representations of every line bundle on $\X$ are trivial. This completes the proof of part (1) of Theorem \ref{thm:P2A5}.

Let $x$ be a singular point of $X$. By Remark \ref{rmk:canonical-pic=cl}, there is a correspondence between Weil divisors on $X$ and line bundles on $\X$. Since every line bundle on $\X$ has trivial residual representations, every Weil divisor on $X$ has trivial image in $\Cl(\hhat\O_{X,x})$. By Theorem \ref{thm:criterion} and Remark \ref{rmk:criterion-zariski-local}, any variety $Y$ which contains a dense open subset isomorphic to a neighborhood of $x$ in $X$ cannot be of the form $U/G$ with $U$ smooth and $G$ a tame finite diagonalizable group.

\end{document}